\documentclass[reqno]{amsart}
\usepackage{amsmath, amssymb, amsthm, graphicx,enumerate}

\usepackage{color}

\definecolor{DehColor}{rgb}{1,0,0}
\definecolor{FroColor}{rgb}{0,0,1}

\long\def\DDD[#1]{\color{red}{#1}}

\font\dixmath=cmsy10
\theoremstyle{plain}
\newtheorem{prop}{Proposition}[section]
\newtheorem{coro}[prop]{Corollary}
\newtheorem*{coro*}{Corollary}
\newtheorem{lemm}[prop]{Lemma} 
\newtheorem{thrm}[prop]{Theorem}
\newtheorem*{thrm1}{Theorem 1}
\theoremstyle{definition}
\newtheorem{defi}[prop]{Definition}

\newtheorem{exam}[prop]{Example}
\newtheorem{rema}[prop]{Remark}
\newtheorem*{propA}{Property A}
\newtheorem*{propC}{Property C}

\numberwithin{equation}{section}

\newenvironment{deflist}{\begin{list}{--}{\setlength{\leftmargin}{15pt}\setlength{\labelwidth}{15pt}}}{\end{list}}
\newenvironment{conditions}{\setlength{\jot}{-1pt}\vspace{-5pt}\subequations\flalign}{\endflalign\endsubequations}

\newcommand{\arXiv}{arXiv:math.GR/0811.3902}
\renewcommand{\aa}[2]{a_{#1, #2}}

\newcommand{\BB}[1]{B_{#1}}
\newcommand{\bidual}{^{\hspace{-0.05em}\raise0.6pt\hbox{$\scriptscriptstyle+$}\hspace{-0.05em}*}}
\newcommand{\BKL}[1]{B_{#1}\bidual}
\newcommand{\BP}[1]{B_{#1}\plusexp}
\newcommand{\br}{\beta}
\newcommand{\brbr}{\gamma}
\newcommand{\brbrr}{\gamma'}

\newcommand{\brbrneg}{\brbr^{\scriptscriptstyle-}}
\newcommand{\brbrpos}{\brbr^{\scriptscriptstyle+}}

\newcommand{\brr}{\beta'}

\newcommand{\brrr}{\beta''}

\newcommand{\brneg}{\br^{\scriptscriptstyle-}}
\newcommand{\brpos}{\br^{\scriptscriptstyle+}}
\newcommand{\brrneg}{\br'^{\scriptscriptstyle-}}
\newcommand{\brrpos}{\br'^{\scriptscriptstyle+}}
\newcommand{\cond}[1]{&\quad\text{-- #1}&}

\newcommand{\ddd}[1]{\delta_{#1}}
\newcommand{\dddd}[2]{\delta_{#1,#2}}

\newcommand{\equal}\equiv
\newcommand{\ff}[1]{\phi_{\hspace{-0.5pt}\raise-1pt\hbox{$\scriptstyle #1$}}}

\newcommand{\find}[1]{e(#1)}
\newcommand{\findi}[1]{f(#1)}

\renewcommand{\ge}{\geqslant}
\def\homs(#1){\underline{#1}}
\newcommand{\ie}{\emph{i.e.}}
\newcommand{\inv}{^{\minus\hspace{-0.1em}1}}
\def\Item(#1){$\quad(#1)$}
\newcommand{\last}[1]{{#1}^{\scriptscriptstyle\mathtt{\#}}}
\newcommand{\Ldots}{...\,} 
\renewcommand{\le}{\leqslant}

\newcommand{\minus}{\mathchoice{-}{-}{\raise0.7pt\hbox{$\scriptscriptstyle-$}\scriptstyle}{-}}
\newcommand{\newintegeri}[2]{
	\expandafter\def\csname #1\endcsname{#2}
	\expandafter\def\csname #1o\endcsname{{#2{-}1}}
	\expandafter\def\csname #1t\endcsname{{#2{-}2}}
	\expandafter\def\csname #1p\endcsname{{#2{+}1}}
	\expandafter\def\csname #1pp\endcsname{{#2{+}2}}}
\newcommand{\newintegerii}[1]{\newintegeri{#1#1}{#1}}

\newcommand{\plus}{\mathchoice{+}{+}{\raise0.7pt\hbox{$\scriptscriptstyle+$}\scriptstyle}{+}}
\newcommand{\plusminus}{\mathchoice{\pm}{\pm}{\raise0.7pt\hbox{$\scriptscriptstyle\pm$}\scriptstyle}{\pm}}
\newcommand{\plusexp}{^{\raise0.8pt\hbox{$\hspace{-0.05em}\scriptscriptstyle+$}}}

\newcommand{\sep}[2]{\widehat\delta_{#1,#2}}
\newcommand{\ShortLex}{\mathtt{ShortLex}}
\newcommand{\sig}[1]{\sigma_{\!#1}} 
\newcommand{\siginv}[1]{\sigma_{\!#1}^{\hspace{-0.05em}\raise0.8pt\hbox{$\scriptscriptstyle-$}\hspace{-0.1em}1}}
\newcommand{\sigpm}[1]{\sigma_{\!#1}^{\hspace{-0.05em}\raise0.8pt\hbox{$\scriptscriptstyle\pm$}\hspace{-0.1em}1}}
\newcommand{\sigg}{\sigma}
\newcommand{\SL}[1]{<_{#1}^*}

\newcommand{\SLL}{<^*}
\newcommand{\SLLe}{\le^*}

\newcommand{\ww}{w}

\newcommand{\www}{\ww'}

\newcommand{\xx}{x}

\newintegeri{indi}{p}
\newintegeri{indii}{q}
\newintegeri{indiii}{r}
\newintegeri{indiv}{s}
\newintegeri{indv}{t}
\newintegeri{indvi}{t'}
\newintegeri{brdi}{b}
\newintegeri{brdii}{c}
\newintegeri{ll}{\ell}
\newintegerii{d}
\newintegerii{h}
\newintegerii{i}
\newintegerii{j}
\newintegerii{k}
\newintegerii{m}
\newintegerii{n}
\newintegerii{p}
\newintegerii{q}
\newintegerii{r}
\newintegerii{s}
\newintegerii{t}

\begin{document} 

\title{The well-ordering of dual braid monoids}
\author{Jean Fromentin} 
\address{Laboratoire de Math\'ematiques Nicolas Oresme,
  UMR 6139 CNRS, Universit\'e de Caen BP 5186, 14032 Caen, France}
\email{jean.fromentin@math.unicaen.fr}
\maketitle

\begin{abstract}
We describe the restriction of the Dehornoy ordering of braids to the 
dual braid monoids introduced by Birman, Ko and Lee:
we give an inductive characterization of the
ordering of the dual braid monoids and compute the
corresponding ordinal type. The proof consists in
introducing a new ordering on the dual braid monoid
using the rotating normal form of
\arXiv, and then proving that this new ordering coincides
with the standard ordering of braids.
\end{abstract}

%
%

It is known since~\cite{Dehornoy:LD} and~\cite{Laver} that the braid group~$\BB\nn$ is left-orderable, by an ordering whose restriction to the positive braid monoid is a well-order.
Initially introduced by complicated methods of self-distributive algebra, the standard braid ordering has then received a lot of alternative constructions 
originating from different approaches---see
\cite{DehornoyDynnikovRolfsenWiest}. However, this ordering remains a
complicated object, and many questions involving it remain open.

Dual braid monoids have been introduced by Birman, Ko, and Lee in~\cite{Birman}.
The dual braid monoid~$\BKL\nn$ is a certain submonoid of the $\nn$-strand braid group~$\BB\nn$.
It is known that the monoid~$\BKL\nn$ admits a Garside structure, where simple elements correspond to non-crossing partitions of~$\nn$---see~\cite{Bessis}.
In particular, there exists a standard normal form associated with this Garside structure, namely the so-called greedy normal form.

The rotating normal form is another normal form
on~$\BKL\nn$ that was introduced in~\cite{F:SR}. It
relies on the existence of a natural embedding
of~$\BKL\nno$ in~$\BKL\nn$ and on the easy
observation that each element of~$\BKL\nn$ admits a
maximal right divisor that belongs to~$\BKL\nno$. The
main ingredient in the construction of the rotating normal
form is the result that each braid~$\br$ in~$\BKL\nno$
admits a unique decomposition
$$\br= \ff\nn^\brdio(\br_\brdi) \cdot ...  \cdot
\ff\nn^2(\br_3) \cdot \ff\nn(\br_2) \cdot \br_1$$  with
$\br_\brdi, \Ldots, \br_1$ in~$\BKL\nno$ such that $\br_\brdi\not=1$ and such that for each $\kk\ge1$, the braid~$\br_\kk$ is the maximal right-divisor of $\ff\nn^{\brdi\minus\kk}(\br_\brdi)\cdot...\cdot\br_\kk$ that lies in $\BKL\nno$.
The sequence $(\br_\brdi, ..., \br_1)$ is then called the
\emph{$\ff\nn$-splitting} of~$\br$. 

The main goal of this paper is to establish the following
simple  connection between the order on~$\BKL\nn$ and the
order on~$\BKL\nno$ through the notion
of~$\ff\nn$-splitting.

\begin{thrm1}
For all braids~$\br, \brbr$ in~$\BKL\nn$ with $\nn\ge2$, the relation $\br
< \brbr$ is true if and only if the $\ff\nn$-splitting
$(\br_\brdi, \Ldots, \br_1)$ of~$\br$ is smaller than the
$\ff\nn$-splitting $(\brbr_{\brdii}, \Ldots, \brbr_1)$ of~$\brbr$ 
with respect to the $\ShortLex$-extension of the ordering
of~$\BKL\nno$, \ie, we have either $\brdi < \brdii$, or $\brdi
= \brdii$ and there exists $\tt$ such that $\br_\tt < \brbr_\tt$ holds and $\br_\kk = \brbr_\kk$ holds for $\brdi \ge\kk > \tt$.
\end{thrm1}

A direct application of Theorem~1 is:

\begin{coro*}
For $\nn\ge2$, the restriction of the braid ordering to~$\BKL\nn$ is
a well-ordering of ordinal type~$\omega^{\omega^\nnt}$.
\end{coro*}

This refines a former result by Laver stating that the
restriction of the braid ordering to~$\BKL\nn$ is a
well-ordering without determining its exact type. 

Another application of Theorem~1 or, more exactly, of
its proof, is a new proof of the existence of the braid
ordering. What we precisely obtain is a new
proof of the result that every nontrivial braid can be
represented by a so-called
$\sigg$-positive or $\sigg$-negative word
(``Property~$\mathbf{C}$'').

The connection between the restrictions of the braid
order to~$\BKL\nn$ and~$\BKL\nno$ via the
$\ff\nn$-splitting is formally similar to the connection
between the restrictions of the braid order to
the Garside monoids~$\BP\nn$ and~$\BP\nno$ via the
so-called $\Phi_\nn$-splitting established
in~\cite{Dehornoy:AF} as an application of Burckel's
approach of~\cite{Burckel:OT,Burckel:WO,Burckel:CO}.
However, there is an important difference, namely that,
contrary to Burckel's approach, our construction requires
no transfinite induction: although intricate in the general
case of $5$~strands and above, our proof remains
elementary. This is an essential advantage of using
the Birman--Ko--Lee generators rather than the Artin
generators. 

The paper is organized as follows.
In Section~\ref{S:Introduction}, we briefly recall the definition of the Dehornoy 
ordering of braids and the definition of the dual braid
monoid. In Section~\ref{S:RotatingOrdering}, we use the $\ff\nn$-splitting to construct a new linear
ordering of~$\BKL\nn$, called the rotating
ordering.
In Section~\ref{S:KeyLemma}, we deduce from the results about $\ff\nn$-splittings established in~\cite{F:SR} the result that certain specific braids are $\sigg$-positive or trivial. Finally, Theorem~1 is proved in Section~\ref{S:ProofOfTheMainResult}.

\section{The general framework}
\label{S:Introduction}

Artin's braid group~$\BB\nn$ is defined for
$\nn\ge2$ by the presentation
\begin{equation}
\label{E:BnPresentation}
\left<\sig1,\Ldots,\sig\nno ; \begin{array}{cl}
\sig\ii\sig\jj\,=\,\sig\jj\sig\ii & \text{for
$|\ii\minus\jj|\ge2$}\\
\sig\ii\sig\jj\sig\ii\,=\,\sig\jj\sig\ii\sig\jj & \text{for
$|\ii\minus\jj|=1$}  \end{array}\right>.
\end{equation}

The submonoid of $\BB\nn$ generated by $\{\sig1,\Ldots,\sig\nno\}$ is denoted by $\BP\nn$.

\subsection{The standard braid ordering}
\label{SS:BraidOrdering}

We recall the construction of the Dehornoy ordering of
braids. By a \emph{braid word} we mean any word on the
letters~$\sig\ii^{\pm1}$.

\begin{defi}
 \label{D:Order}
~
\begin{deflist}
\item A braid word~$\ww$ is called \emph{$\sig\ii$-positive} (\emph{resp. $\sig\ii$-negative}) 
if $\ww$ contains at least one~$\sig\ii$ (\emph{resp.} at
least one~$\siginv\ii$), no~$\siginv\ii$ (\emph{resp.} no
$\sig\ii$), and no letter~$\sigpm\jj$ with $\jj>\ii$.
\item A braid $\br$ is said to be \emph{$\sig\ii$-positive} (\emph{resp. $\sig\ii$-negative}) if, among the braid words
  representing~$\br$, at least one is $\sig\ii$-positive (\emph{resp.} $\sig\ii$-negative).
\item A braid $\br$ is said to be \emph{$\sigg$-positive} (\emph{resp. $\sigg$-negative}) if it is $\sig\ii$-positive (resp. $\sig\ii$-negative) for some $\ii$.
\item For $\br$, $\brbr$ braids, we declare that $\br<\brbr$ is true if the braid $\br\inv\,\brbr$ is $\sigg$-positive.
\end{deflist}
\end{defi}

By definition of the relation $<$, every $\sigg$-positive braid $\br$ satisfy  $1<\br$.
Then, every braid of $\BP\nn$ except $1$ is lager than $1$.

\begin{exam}
Put $\br=\sig2$ and $\brbr=\sig1\,\sig2$.
Let us show that $\br$ is $<$-smaller than~$\brbr$.
The quotient $\br\inv\,\brbr$ is represented by the word $\siginv2\,\sig1\,\sig2$.
Unfortunately, the latter word is neither $\sig2$-positive (since it contains $\siginv2$), nor $\sig2$-negative (since it contains~$\sig2$), nor $\sig1$-positive and $\sig1$-negative (since it contains a letter $\sigpm2$).
However, the word $\siginv2\,\sig1\,\sig2$ is equivalent to $\sig1\,\sig2\,\siginv1$, which is $\sig2$-positive.
Then, the braid~$\br\inv\,\brbr$ is $\sig2$-positive and the relation $\br<\brbr$ holds.
\end{exam}

\begin{thrm} \cite{Dehornoy:LD}
  \label{P:Order}
  For each~$\nn\ge2$, the relation~$<$ is a linear ordering on~$\BB\nn$
  that is invariant under left multiplication.
\end{thrm}

\begin{rema}
In this paper, we use the flipped version of
  the braid ordering~\cite{DehornoyDynnikovRolfsenWiest}, in which one takes into
account the   generator~$\sig\ii$ with greatest index, and not the original version,
  in which one considers the generator with lowest index. This choice
  is necessary here, as we need that $\BKL\nno$ is an
initial   segment of~$\BKL\nn$. That would not be true if we were
considering the lower version of the ordering. We recall that the flip automorphism~$\Phi_\nn$
that maps~$\sig\ii$ and~$\sig{\nn-\ii}$ for each~$\ii$
exchanges the two version, and, therefore, the  properties of both orderings are identical.
\end{rema}

Following~\cite{Dehornoy:LD}, we recall that
Theorem~\ref{P:Order} relies on two results:

\begin{propA}
Every $\sig\ii$-positive braid is nontrivial.
\end{propA}

\begin{propC}
 Every braid is either trivial or $\sigg$-positive or $\sigg$-negative.
\end{propC}

In the sequel, as we shall prove that Property~$\mathbf{C}$ is a consequence of Theorem~$1$, we never use Theorem~\ref{P:Order}, \ie,
we never use the fact that the relation~$<$ of Definition~\ref{D:Order} is a linear ordering.
The only properties of~$<$ we shall use are the following trivial facts---plus, but
exclusively for the corollaries, Property~$\mathbf{A}$.

\begin{lemm}
\label{L:Transitive}
 The relation $<$ is transitive, and invariant under left-multiplication.
\end{lemm}

The ordering $<$ on $\BB\nn$ admits lots of properties but it is not a well-ordering.
For instance, 
$\sigma_\nno\inv,\sigma_\nno^{-2},\Ldots,\sigma_\nno^{-\kk},\Ldots$ is an infinite
descending sequence. However, R.\,Laver
proved the following result:

\begin{thrm}\cite{Laver}
\label{P:WellOrder}
Assume that  $M$ is a submonoid of $\BB\infty$ generated 
by a finite number of braids, each of which is a
conjugate of some braid~$\sig\ii$. Then the restriction
of~$<$ to~$M$ is a well-ordering.
\end{thrm}

The positive braid monoid~$\BP\nn$ satisfies the
hypotheses of Theorem~\ref{P:WellOrder}. Therefore, the
restriction of the braid ordering to~$\BP\nn$ is a
well-ordering. However, Laver's proof of Theorem~\ref{P:WellOrder} leaves the
determination of the isomorphism type of $(M,
<)$ open. In the case of the monoid~$\BP\nn$, the
question was solved by S.\,Burckel:

\begin{prop}\cite{Burckel:WO}
For each $\nn\ge2$, the order type of $(\BP\nn,<)$ is
the ordinal~$\omega^{\omega^\nnt}$.
\end{prop}

\subsection{Dual braid monoids}
\label{SS:DualBraidMonoids}

In this paper, we consider another monoid of which~$\BB\nn$ is a group of fractions, namely the
Birman-Ko--Lee monoid, also called \emph{the dual braid
monoid}.

\begin{defi}
For $1\le\indi<\indii$, put
\begin{equation}
\label{E:DualGenerator}
\aa\indi\indii = \sig\indi ... \sig\indiit \, \sig\indiio \, 
\siginv\indiit ... \siginv\indi.
\end{equation}
Then the \emph{dual braid monoid~$\BKL\nn$} is the
submonoid of $\BB\nn$ generated by the
braids~$\aa\indi\indii$ with $1\le\indi<\indii\le\nn$; the
braids~$\aa\indi\indii$ are called the
\emph{Birman--Ko--Lee generators}.
\end{defi}

\begin{rema}
In~\cite{Birman}, the braid $\aa\indi\indii$ is defined to be $\sig\indiio... \sig\indip \, \sig\indi \, \siginv\indip ... \siginv\indiio$.
Both options lead to isomorphic monoids, but our choice is the only one that naturally leads to the suitable embedding of~$\BKL\nno$ into~$\BKL\nn$.
\end{rema}

By definition, we have $\sig\indi = \aa\indi\indip$,
so the dual braid monoid~$\BKL\nn$ includes the positive
braid monoid~$\BP\nn$. The inclusion is proper for
$\nn\ge3$: the braid~$\aa13$, which belongs
to~$\BKL\nn$ by definition,  does not belong to
$\BP\nn$.

The following notational convention will be useful in
the sequel. For $\pp \le \qq$, we write $[\indi,\indii]$
for the interval $\{\indi,\Ldots,\indii\}$ of
$\mathbb{N}$. We say that $[\indi,\indii]$ is \emph{nested} in
$[\indiii,\indiv]$ if we have
$\indiii<\indi<\indii<\indiv$. The following results
were proved by Birman, Ko, and Lee.

\begin{prop}\cite{Birman}

$(i)$ In terms of the generators~$\aa\indi\indii$, the
monoid~$\BKL\nn$ is presented by  the relations
\begin{align}
\label{E:DualCommutativeRelation}
\aa\indi\indii\aa\indiii\indiv&= \aa\indiii\indiv\aa\indi\indii \text{\quad for $[\indi, \indii]$ and $[\indiii, \indiv]$ disjoint or nested},\\
\label{E:DualNonCommutativeRelation}
\aa\indi\indii\aa\indii\indiii&= \aa\indii\indiii\aa\indi\indiii =\aa\indi\indiii\aa\indi\indii \text{\quad for $1 \le \indi<\indii<\indiii\le\nn$}.
\end{align}

$(ii)$ The monoid~$\BKL\nn$ is a Garside monoid
with Garside element~$\ddd\nn$ defined by
\begin{equation}
\label{E:DualGarsideElement}
\ddd\nn=\aa12\,\aa23\,...\,\aa\nno\nn\quad
(=\sig1\,\sig2\,...\,\sig\nno).
\end{equation}
\end{prop}

Garside monoids are defined for instance
in~\cite{DehornoyParis} or~\cite{Dehornoy:GG}.
In every Garside
monoid, conjugating by the Garside element defines an
automorphism (.
 In the case of~$\BKL\nn$, the
automorphism~$\ff\nn$
defined by $\ff\nn(\br)=\ddd\nn\,\br\,\ddd\nn\inv$ has order
$\nn$, and its action on the generators~$\aa\indi\indii$ is
as follows.

\begin{lemm}
For all $\indi,\indii$ with $1\le\indi<\indii\le\nn$, we have
\begin{equation}
\label{E:GarsideAutoOnLetter}
\ff\nn(\aa\indi\indii)=
\begin{cases}
\aa\indip\indiip & \text{for $\indii\le\nno$,}\\
\aa1\indip & \text{for $\indii=\nn$.}
\end{cases}
\end{equation}
\end{lemm}

The relations~\eqref{E:GarsideAutoOnLetter} show that
the action of~$\ff\nn$ is  similar to a rotation. Note that the
relation $\ff\nn(\aa\indi\indii)=\aa\indip\indiip$ always holds  provided
indices are taken mod $\nn$ and possibly switched, so
that, for instance, $\aa\indip\nnp$ means
$\aa1\indip$.

For every braid~$\br$ in~$\BKL\nn$ and
every~$\kk$, the definition of~$\ff\nn$ implies the
relation
\begin{equation}
\label{E:PushDdd}
\ddd\nn\,\ff\nn^\kk(\br)=\ff\nn^\kkp(\br)\,\ddd\nn,
\quad\ddd\nn\inv\,\ff\nn^\kk(\br)=\ff\nn^\kko(\br)\,
\ddd\nn\inv.
\end{equation}

By definition, the braid $\aa\indi\indii$ is the
conjugate of $\sig\indiio$ by the braid
$\sig\indi\Ldots\sig\indiit$. Braids of the latter type play
an important role in the sequel, and we give them a
name.

\begin{defi}
For $\indi\le\indii$, we put 
\begin{equation}
\label{E:DecompAWithD}
\dddd\indi\indii = \aa\indi\indip\,\aa\indip\indipp\,...\,
\aa\indiio\indii \ (\  = \sig\indi\sig\indip \,...\, \sig\indiio
). 
\end{equation}
Note that the Garside element~$\ddd\nn$ of $\BKL\nn$ is equal to $\dddd1\nn$ and that the braid $\dddd\indi\indi$ is the trivial one, \ie, is  the braid $1$.
\end{defi}

With this notation, we easily obtain
\begin{align}
\label{E:AConjugateD}
\aa\indi\indii &= \dddd\indi\indii\, \dddd\indi{\indii-1}
\inv = \dddd\indi\indiio\ \sig\iio\ \dddd\indi\indiio\inv
&&\qquad\textrm{for $\indi<\indii$},\\
\label{E:RelationD:i}\aa\indi\indii&=\dddd\indip\indii\inv\ 
\dddd\indi\indii&&\qquad\textrm{for $\indi<\indii$},\\
\label{E:RelationD:ii}\dddd\indi\indiii&=\dddd\indi\indii\ \dddd\indii\indiii&&\qquad\textrm{for $\indi\le\indii\le\indiii$},\\
\label{E:RelationD:iii}\dddd\indi\indii\ \dddd{\indiii}{\indiv}&=\dddd{\indiii}{\indiv}\ \dddd\indi\indii&&\qquad\textrm{for $\indi\le\indii<\indiii\le\indiv$}.
\end{align}

Relations \eqref{E:AConjugateD}, \eqref{E:RelationD:ii}, and
\eqref{E:RelationD:iii} are direct consequences of the definition of
$\dddd\indi\indii$. Relation~\eqref{E:RelationD:i} is obtained by
using an induction on~$\indi$ and the equality 
$\dddd\indip\indii\inv\,\dddd\indi\indii=\sig\indi\,
\dddd\indipp\indii\inv\,\dddd\indip\indii\,  \siginv
\indi$, which holds for $\indi<\indii$.

\subsection{The restriction of the braid ordering to the dual braid
monoid}
\label{SS:RestrictionBraidOrderingDualBraidMonoid}

The aim of this paper is to describe the restriction of the
braid ordering of Definition~\ref{D:Order} to the dual braid
monoid~$\BKL\nn$. The initial observation is:

\begin{prop}\cite{Laver}
For each~$\nn\ge2$, the restriction of the braid ordering to
the monoid~$\BKL\nn$ is a well-ordering.
\end{prop}

\begin{proof}
By definition, the braid~$\aa\indi\indii$ is a conjugate
of the braid~$\sig\indiio$. So $\BKL\nn$ is generated by finitely many
braids, each of which is a conjugate of some~$\sig\ii$. By Laver's
Theorem (Theorem~\ref{P:WellOrder}), the restriction of~$<$
to~$\BKL\nn$ is a well-ordering.
\end{proof}

As in the case of~$\BP\nn$, Laver's Theorem, which is an
non-effective result based on the so-called Higman's Lemma
\cite{Higman}, leaves the determination of the isomorphism type
of~$<\,\upharpoonright_{\BKL\nn}$ open. This is the question we
shall address in the sequel.

Before introducing our specific methods, let us begin with some
easy observations.

\begin{lemm}
Every braid~$\br$ in $\BKL\nn$ except $1$ satisfies~$\br >1$.
\end{lemm}

\begin{proof}
By definition, the braid $\aa\indi\indii$ is
$\sigg$-positive---actually it is $\sig\indiio$-positive in the sense of
Definition~\ref{D:Order}.
\end{proof}

\begin{lemm}
\label{X:OrdreA}
For each~$\nn\ge2$, we have
\begin{equation}
\label{E:StandarOrderofA}
 1<\aa12<\aa23<\aa13<\Ldots<\aa1\nno<\aa\nno\nn<\aa\nnt\nn
<\Ldots<\aa1\nn.
\end{equation}
\end{lemm}

\begin{proof}
We claim that $\aa\indi\indii<\aa\indiii\indiv$ holds 
if and only if we have either $\indii<\indiv$, or
$\indii=\indiv$ and
$\indi>\indiii$. Assume first $\indii<\indiv$. Then, the braid
$\aa\indi\indii\inv$ is
$\sig\indiio$-negative while $\aa\indiii\indiv$ is
$\sig\indivo$-positive with $\indii<\indiv$, hence the quotient
$\aa\indi\indii\inv\,\aa\indiii\indiv$ is $\sig\indivo$-positive, which
implies $\aa\indi\indii<\aa\indiii\indiv$. Assume now $\indii=\indiv$
and $\indi>\indiii$. Then, by relation~\eqref{E:AConjugateD}, the
quotient~$\aa\indi\indii\inv\,\aa\indiii\indiv$ is equal to
$\dddd\indi\indivo\,\dddd\indi\indiv\inv\,\dddd\indiii\indiv\,\dddd\indiii\indivo\inv$.
Applying Relation~\eqref{E:RelationD:ii} on~$\dddd\indiii\indiv$, we
obtain 
\[
\aa\indi\indii\inv\,\aa\indiii\indiv=\dddd\indi\indivo\,\dddd\indi\indiv\inv\,\dddd\indiii\indio\,\dddd\indio\indiv\,\dddd\indiii\indivo\inv.
\]
Then, by applying Relation~\eqref{E:RelationD:iii} on $\dddd\indi\indiv\inv\,\dddd\indiii\indio$, we obtain
\[
\aa\indi\indii\inv\,\aa\indiii\indiv=\dddd\indi\indivo\,\dddd\indiii\indio\,\dddd\indi\indiv\inv\,\dddd\indio\indiv\,\dddd\indiii\indivo\inv.
\]
Finally, Relation~\eqref{E:RelationD:i} on $\dddd\indi\indiv\inv\,\dddd\indio\indiv$ implies 
\[
\aa\indi\indii\inv\,\aa\indiii\indiv=\dddd\indi\indivo\,\dddd\indiii\indio\,\aa\indio\indiv\,\dddd\indiii\indivo\inv.
\]
The braid $\aa\indio\indiv$ is $\sig\indivo$-positive, while the braids $\dddd\indi\indivo$, $\dddd\indiii\indio$ and $\dddd\indiii\indivo\inv$ are  $\sig\indv$-positive or $\sig\indv$-negative for $\indv<\indivo$.
Hence the quotient $\aa\indi\indii\inv\,\aa\indiii\indiv$ is $\sig\indivo$-positive, which
implies $\aa\indi\indii<\aa\indiii\indiv$.

As $<$ is a linear ordering, this is enough to conclude, \ie, the
implications we proved above are equivalences.
\end{proof}

\subsection{The rotating normal form}
\label{SS:TheRotatingNormalForm}

In this section we briefly recall the construction of the rotating normal form of~\cite{F:SR}.

\begin{defi}
For $\nn\ge3$ and $\br$ a braid of~$\BKL\nn$.
The maximal braid~$\br_1$ lying in~$\BKL\nno$ that right-divides the braid~$\br$ is called the 
\emph{$\BKL\nno$-tail} of~$\br$.
\end{defi}

\begin{prop}\cite{F:SR}
\label{P:Splitting}
Assume $\nn\ge3$ and that $\br$ is a nontrivial braid in~$\BKL\nn$.
Then there exists a unique sequence $(\br_\brdi,\Ldots,\br_1)$ in $\BKL\nno$ satisfying $\br_\brdi \not= 1$ and
\begin{conditions}
\label{E:SplittingDecomposition}
\cond{$\br=\ff\nn^\brdio(\br_\brdi)\cdot...\cdot\ff\nn(\br_2)\cdot\br_1,$}\\
\label{E:SplittingCondition}
\cond{for each $\kk\ge1$, the~$\BKL\nno$-tail of~$\ff\nn^{\brdi\minus\kk}(\br_\brdi)\cdot...\cdot\ff\nn(\br_\kkp)$ is trivial.}
\end{conditions}

\end{prop}

\begin{defi}
\label{D:Splitting}
The unique sequence $(\br_\brdi,\Ldots, \br_1)$ of braids introduced in Proposition~\ref{P:Splitting} is called the \emph{$\ff\nn$-splitting of~$\br$}. 
Its length, \ie, the parameter~$\brdi$, is called the \emph{$\nn$-breadth}
of~$\br$.
\end{defi}

The idea of the~$\ff\nn$-splitting is very simple: starting 
with a braid~$\br$ of~$\BKL\nn$, we extract the maximal right-divisor that lies in~$\BKL\nno$, \ie, that leaves the~$\nn$th
strand unbraided, then we extract the maximal right-divisor of the
remainder that leaves the first strand unbraided, and so on
rotating by~$2\pi/\nn$ at each step---see Figure~\ref{F:Splitting}.

\setlength{\unitlength}{1mm}
\begin{figure}[htb]
\begin{picture}(104,28)
\put(0,0){\includegraphics{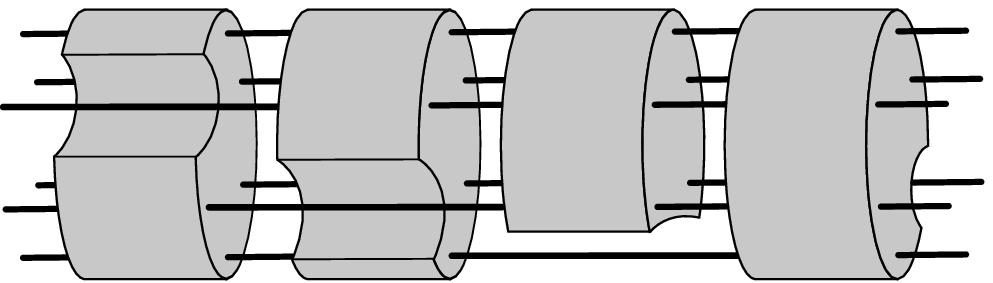}}
\put(80,13){$\br_1$} \put(53,13){$\ff6(\br_2)$}
\put(31,17){$\ff6^2(\br_3)$} \put(8,8){$\ff6^3(\br_4)$}
{\footnotesize
\put(101.55,9.5){$6$} \put(101.5,20){$5$}
\put(100,25){$4$} \put(97.5,16.5){$3$}
\put(97.6,6.5){$2$} \put(100,2){$1$}}
\end{picture}
\caption{{\sf \smaller The $\ff6$-splitting of a braid of~$\BKL6$.  
Starting from the right, we extract the maximal right-divisor that
keeps the sixth strand unbraided, then rotate by $2\pi/6$ and
extract the maximal right-divisor that keeps the first strand
unbraided, etc.}}
\label{F:Splitting}
\end{figure}

As the notion of a~$\ff\nn$-splitting is fundamental in this paper, we give examples.

\begin{exam}
\label{X:GeneratorSplitting}
Let us determine the~$\ff\nn$-splitting of the generators
of~$\BKL\nn$,
\ie,
of $\aa\indi\indii$ with
$1\le\indi<\indii\le\nn$. For $\indii\le\nno$, the generator~$\aa\indi\indii$
belongs to~$\BKL\nno$, then its $\ff\nn$-splitting is $(\aa\indi\indii)$. 
Next, as
$\aa\indi\nn$ does not lie in~$\BKL\nno$, the rightmost entry in
its~$\ff\nn$-splitting is trivial. As we have
$\ff\nn\inv(\aa\indi\nn)=\aa\indio\nno$ for $\indi\ge2$, the~$\ff\nn$-splitting
of~$\aa\indi\nn$ with $\indi\ge2$ is $(\aa\indio\nno,1)$.  
Finally, the braids
$\aa1\nn$ and $\ff\nn\inv(\aa1\nn)=\aa\nno\nn$ do not lie
in~$\BKL\nno$, but
$\ff\nn^{\minus2}(\aa1\nn)=\aa\nnt\nno$ does. So the~$\ff\nn$-splitting
of~$\aa1\nn$ is $(\aa\nnt\nno,1,1)$.  To summarize, the~$\ff\nn$-splitting
of~$\aa\indi\indii$ is
\begin{equation}
\label{E:GeneratorSplitting}
\begin{cases}
(\aa\indi\indii)&\text{for $\indi<\indii\le\nno$},\\
(\aa\indio\nno,1)&\text{for $2\le\indi$ and $\indii=\nn$},\\
(\aa\nnt\nno,1,1)&\text{for $\indi=1$ and $\indii=\nn$}.
\end{cases}
\end{equation}
\end{exam}

By Relation~\eqref{E:GarsideAutoOnLetter}, the application $\ff\nn$ maps each braid~$\aa\indi\indii$ to another similar braid~$\aa\indiii\indiv$.
Using this remark, we can consider the alphabetical homomorphism, still denoted~$\ff\nn$, that maps the letter~$\aa\indi\indii$ to the corresponding letter~$\aa\indiii\indiv$, and extends to word on the letter $\aa\indi\indii$.
Note that, in this way, if the word~$\ww$ on the letter $\aa\indi\indii$ represents the braid~$\br$, then
$\ff\nn(\ww)$ represents~$\ff\nn(\br)$.

We can now recursively define a distinguished expression for each
braid of~$\BKL\nn$ in terms of the generators~$\aa\indi\indii$.

\begin{defi}~
\begin{deflist}
\item For $\br$ in~$\BKL2$, the \emph{$\ff2$-normal form} of~$\br$ is defined to be the unique word~$\aa12^\kk$ that represents~$\br$.
\item For $\nn \ge 3$ and $\br$ in~$\BKL\nn$, the \emph{$\ff\nn$-normal form} of~$\br$ is defined to be the word
$\ff\nn^\brdio(\ww_\brdi) \,... \,\ww_1$ where, for each $\kk$, the word $\ww_\kk$ is the $\ff\nno$-normal form of~$\br_\kk$ and where $(\br_\brdi,\Ldots,\br_1)$ is the~$\ff\nn$-splitting of~$\br$.
\end{deflist}
\end{defi}

As the~$\ff\nn$-splitting of a braid~$\br$ lying in~$\BKL\nno$ is the length $1$ sequence~$(\br)$, the~$\ff\nn$-normal form and $\ff\nno$-normal form of~$\br$ coincide.
Therefore, we can drop the subscript in the~$\ff\nn$-normal
form. From now on, we call \emph{rotating normal form}, or
simply \emph{normal form}, the expression so obtained.

As each braid is represented by a unique normal word, we
can unambiguously use the syntactical properties of its normal form.

We conclude this introductory section with some syntactic
constraints involving $\ff\nn$-splittings and normal words, These
results are borrowed from~\cite{F:SR}.

\begin{defi}
\label{D:Last}
For $\br$ in~$\BKL\nn$, the \emph{last letter} of~$\br$,
denoted~$\last\br$, is defined to be the last latter in the normal
form of~$\br$.
\end{defi}

\begin{lemm} \cite{F:SR}
\label{L:LastLetter}
Assume that $(\br_\brdi,\Ldots,\br_1)$ is a $\ff\nn$-splitting.

$(i)$ For $\kk\ge 2$, the letter~$\last{\br}_\kk$ has the form
$\aa{..}\nno$, unless $\br_\kk$ is trivial;

$(ii)$ For $\kk\ge3$, the braid $\br_\kk$ is different from $1$;

$(iii)$ For $\kk\ge2$, if the normal form of~$\br_\kk$ is $\www\,\aa\nnt\nno$  with
$\www\not=\varepsilon$ (the empty word), then the last letter
of~$\www$ has the form $\aa{..}\nno$.
\end{lemm}

\section{The rotating ordering}
\label{S:RotatingOrdering}

As explained above, we aim at proving results about the restriction of 
the braid ordering~$<$ to the dual braid monoid~$\BKL\nn$. We
shall do it indirectly, by first introducing an auxiliary ordering~$\SLL$,
and eventually proving that the latter coincides with the original braid
ordering. 

\subsection{Another ordering on~$\BKL\nn$}
\label{SS:Construction}

Using the $\ff\nn$-splitting of Definition~\ref{D:Splitting}, 
every braid of $\BKL\nn$ comes associated with a distinguished finite
sequence of braids belonging to $\BKL\nno$. In this way, every ordering
on~$\BKL\nno$ can be extended to an ordering on~$\BKL\nn$ using a
lexicographic extension. Iterating the process, we can start from the
standard ordering on~$\BKL2$, \ie, on natural numbers, and
recursively define a linear ordering on~$\BKL\nn$.

We recall that, if $(A,\prec)$ is an ordered set, a finite sequence~$\ss$ in~$A$ is called $\mathtt{ShortLex}$-smaller than another finite sequence~$\ss'$ if the length of~$\ss$ is smaller than that of~$\ss'$, or if both lengths are equal and $\ss$ is lexicographically $\prec$-smaller than~ $\ss'$, \emph{i.e.}, when both sequences are read starting from the left, the first entry in~$\ss$ that does not coincide with its counterpart in~$\ss'$ is $\prec$-smaller.

\begin{defi}
\label{D:SLn}
For $\nn \ge 2$, we recursively define a relation~$\SL\nn$
on~$\BKL\nn$ as follows:
\begin{deflist}
\item For $\br, \brbr$ in~$\BKL2$, we declare that $\br\SL2\brbr$ is true  
for $\br\!=\!\aa12^\brdi$ and $\brbr\!=\!\aa12^{\brdii}$ with~$\brdi<\brdii$;
\item For $\br, \brbr$ in~$\BKL\nn$ with $\nn\ge3$, we declare that $\br\SL\nn\brbr$ is true  if the~$\ff\nn$-splitting
of~$\br$ is smaller than the $\ff\nn$-splitting of~$\brbr$ for the
$\ShortLex$-extension of~$\SL\nno$.
\end{deflist}
\end{defi}

\begin{exam}
\label{X:OrdreRA}
As was seen in Example~\ref{X:GeneratorSplitting}, the $\nn$-breadth of $\aa\indi\indii$ with $\indii\le\nno$ is $1$ while the $\nn$-breadth of $\aa\indii\nn$ is $2$ for $\indi\not=1$ or $3$ for $\indi=1$.
An easy induction on~$\nn$ gives $\aa\indi\indii\SL\nn\aa\indiii\indiv$ whenever $\indii<\indiv\le\nn$ holds.
Then, one establishes
$$
1\SL\nn\aa12\SL\nn\aa23\SL\nn\aa13\SL\nn\aa34\SL\nn\aa24\SL\nn\aa14\SL\nn...\SL\nn\aa\nno\nn\SL\nn...\SL\nn\aa1\nn.
$$
\end{exam}

We observe that, according to Lemma~\ref{X:OrdreA} and
Example~\ref{X:OrdreRA}, the relations $<$ and $\SL\nn$ agree on the
generators of~$\BKL\nn$.

\begin{prop}
\label{P:PropOfSL}
For $\nn\ge 2$, the relation~$\SL\nn$ is a well-ordering on~$\BKL\nn$. For each braid~$\br$, the immediate
$\SL\nn$-successor of~$\br$ is $\br\, \aa12$, \ie, $\br\, \sig1$.
\end{prop}

\begin{proof}
The ordered monoid $(\BKL2,\SL2)$ is isomorphic to $\mathbb{N}$ with the usual ordering, which is a well-ordering.
As the $\ShortLex$-extension of a well-ordering is itself a 
well-ordering---see \cite{Levy}---we inductively deduce that
$\SL\nn$ is a well-ordering. 

The result about successors immediately follows from the fact that,
if the $\ff\nn$-splitting of~$\br$ is
$(\br_\pp,\Ldots,\br_1)$, then the $\ff\nn$-splitting of $\br\aa12$ is
$(\br_\pp,\Ldots,\br_1\aa12)$.
\end{proof}

The connection between the ordering~$\SL\nno$ and the restriction 
of~$\SL\nn$ to~$\BKL\nno$ is simple:  $\BKL\nno$ is an initial
segment of~$\BKL\nn$.

\begin{prop}
\label{P:Sep2}
For $\nn \ge 3$, the monoid~$\BKL\nno$ is the initial segment of~$(\BKL\nn, \SL\nn)$ determined by~$\aa\nno\nn$, \ie, we have
$\BKL\nno=\{\br\in\BKL\nn\,\mid\,\br\SL\nn\aa\nno\nn\}$. Moreover the braid $\aa\nno\nn$ is the smallest of $\nn$-breadth $2$.
\end{prop}

\begin{proof}
First, by construction, every braid~$\br$ of~$\BKL\nno$ has
$\nn$-breadth $1$, whereas, by~\eqref{E:GeneratorSplitting},  the $\nn$-breadth
of~$\aa\nno\nn$ is~$2$. So, by definition, $\br \SL\nn\aa\nno\nn$
holds.

Conversely, assume that $\br$ is a braid of~$\BKL\nn$ that
satisfies $\br \SL\nn \aa\nno\nn$. As the $\nn$-breadth of~$\aa\nno\nn$ is~$2$, the hypothesis
$\br \SL\nn \aa\nno\nn$ implies that the $\nn$-breadth of~$\br$ is
at most~$2$. We shall prove, using induction on~$\nn$, that $\br$ has
$\nn$-breadth at most~$1$, which, by construction, implies that
$\br$ belongs to~$\BKL\nno$.

Assume first $\nn=3$. By definition, every $\ff3$-splitting
of length $2$ has the form $(\aa12^\brdi,\aa12^\brdii)$ with
$\brdi\neq0$. The $\ShortLex$-least such sequence is~$(\aa12,
1)$, which turns out to be the $\ff3$-splitting of~$\aa23$. Hence
$\aa23$ is the $\SL3$-smallest element of~$\BKL3$ with
$3$-breadth equal to~$2$, and $\br \SL3 \aa23$ implies $\br \in
\BKL2$.

Assume now $\nn>3$. Assume for a contradiction that the $\nn$-breadth of~$\br$ is~$2$. Let $(\br_2, \br_1)$ be the
$\ff\nn$-splitting of~$\br$. As the $\ff\nn$-splitting of~$\aa\nno\nn$ is $(\aa\nnt\nno,1)$, and $\br_1 \SL\nno 1$ is
impossible, the hypothesis $\br \SL\nn \aa\nno\nn$ implies $\br_2 \SL\nno \aa\nnt\nno$. By induction hypothesis, this implies that
$\br_2$ lies in~$\BKL\nnt$, hence $\ff\nn(\br_2)$ lies in~$\BKL\nno$. This contradicts 
Condition~\eqref{E:SplittingCondition}: a sequence $(\br_2,
\br_1)$ with $\br_2 \SL\nno \aa\nnt\nno$ cannot be the $\ff\nn$-splitting of a braid
of~$\BKL\nn$. So the hypothesis that
$\br$ has $\nn$-breadth~$2$ is contradictory, and $\br$ necessarily
lies in~$\BKL\nno$.
\end{proof}

Building on the compatibility result of Proposition~\ref{P:Sep2}, we hereafter drop
the subscript in~$\SL\nn$ and simply write~$\SLL$. Note that $\SLL$ is
actually a linear  order (and even a well-ordering) on~$\BKL\infty$, the inductive
limit of the monoids~$\BKL\nn$ with respect to the canonical embedding of~$\BKL\nno$
into~$\BKL\nn$.

\subsection{Separators}
\label{SS:SeparatorBraids}

By definition of $\SLL$, for $\brdi < \brdii$, every braid in~$\BKL\nn$ that has $\nn$-breadth~$\brdi$ is $\SLL$-smaller than every braid that has 
$\nn$-breadth~$\brdii$. As the ordering~$\SLL$ is a well-ordering, there must exist, for each~$\brdi$, a $\SLL$-smallest braid with
$\nn$-breadth~$\brdi$. These braids, which play the role of separators
for~$\SLL$, are easily identified. They will play an important role in the sequel. 

Proposition~\ref{P:Sep2} says that the least upper bound of the braids with
$\nn$-breadth~$1$ is $\aa\nno\nn$. From $\nn$-breadth~$2$, a
periodic pattern appears.

\begin{defi}
\label{D:Sep}
For $\nn\ge3$ and $\brdi\ge1$, we put 
$\sep\nn\brdi{=}\ff\nn^{\brdip}(\aa\nnt\nno)\cdot...\cdot\ff\nn^2(\aa\nnt\nno)$.
\end{defi}

For instance, we find $\sep64=\ff6^5(\aa45)\cdot\ff6^4(\aa45)\cdot\ff6^3(\aa45)
\cdot\ff6^2(\aa45)$, whence 
\linebreak
$\sep64=  \aa34\,\aa23\,\aa12\,\aa16$,
and, similarly, 
$\sep53=\aa23\,\aa12\,\aa15$.

\begin{prop}
\label{P:Sep}
For all~$\nn \ge 3$ and $\brdi\ge 1$,

$(i)$ the $\ff\nn$-splitting of~$\sep\nn\brdi$ is the length~$\brdipp$ sequence
$(\aa\nnt\nno,...,\aa\nnt\nno,1,1)$;

$(ii)$ we have $\sep\nn\brdi=\ddd\nn^\brdi\ddd\nno^{-\brdi}$;

$(iii)$ the braid~$\sep\nn\brdi$ is the $\SLL$-smallest braid in~$\BKL\nn$ that has $\nn$-breadth~$\brdipp$---hence it is the least upper
bound of all braids of $\nn$-breadth~$\le \brdip$.
\end{prop}

\begin{proof}
$(i)$ First, we observe that there exists no relation $\aa\nno\nn\,\aa\nnt\nno=...$
in the presentation of the monoid~$\BKL\nn$. Then the word $\aa\nno\nn\aa\nnt\nno$,
which is equal to $\ff\nn(\aa\nnt\nno)\,\aa\nnt\nno$, is alone in its
equivalence class under the relation of~$\BKL\nn$. Similarly, the word
$\ff\nn^{\brdio}(\aa\nnt\nno)\cdot...\cdot\aa\nnt\nno$, called $\ww$, is alone in its
equivalence class, as no relation of the presentation may be applied to any
length~$2$ subword of this word. As $\ff\nn$ is an isomorphism, the same result holds
for the word $\ff\nn^2(\ww)$, which represents the braid
$\sep\nn\brdi$. As the braid $\sep\nn\brdi$ is
represented by a normal word, we deduce that $\ff\nn^2(\ww)$ is the
normal word representing~$\sep\nn\brdi$, \ie, it is its normal form.

$(ii)$ We use an induction on $\brdi$.
Relation~\eqref{E:AConjugateD} implies $\aa1\nn=\ddd\nn\ddd\nno\inv$.
Using~\eqref{E:GarsideAutoOnLetter}, we deduce
\[
\sep\nn1=\ff\nn^2(\aa\nnt\nno)=\aa1\nn=\ddd\nn\,\ddd\nno\inv.
\]
Assume now $\brdi\ge 2$.
By definition, we have $\sep\nn\brdi=\ff\nn^\brdip(\aa\nnt\nno)\,\sep\nn\brdio$.
Then, using the induction hypothesis we have 
\[
\sep\nn\brdi=\ff\nn^\brdip(\aa\nnt\nno)\,\ddd\nn^\brdio\,\ddd\nno^{-\brdi+1}.
\]
Pushing $\ddd\nn^\brdio$ to the left using \eqref{E:PushDdd}, we find:
\[
 \sep\nn\brdi=\ddd\nn^\brdio\,\ff\nn^2(\aa\nnt\nno)\,
\ddd\nno^{-\brdi+1}.
\]
Relation~\eqref{E:GarsideAutoOnLetter} implies $\ff\nn^2(\aa\nnt\nno) = \aa1\nn$.
So, using~\eqref{E:AConjugateD}, we finally obtain
\[
\sep\nn\brdi=\ddd\nn^\brdio\,\ddd\nn\,\ddd\nno\inv\,\ddd\nno^{-\brdi+1}=
\ddd\nn^\brdi\,\ddd\nno^{-\brdi}.
\]

$(iii)$ Let $(\br_\brdipp,...,\br_1)$ be the $\ff\nn$-splitting of a braid~$\br$
that lies in~$\BKL\nn$ and satisfies $\br\SLLe\sep\nn\brdi$. By definition of
$\SLL$, we have $\br_\brdipp\SLLe\aa\nnt\nno$. By
Lemma~\ref{L:LastLetter}$(i)$ and~$(ii)$, the $\BKL\nnt$-tail of~$\br_\brdipp$ is trivial,
hence its $\nno$-breadth is at least~$2$.  Then Proposition~\ref{P:Sep2} implies
$\br_\brdipp = \aa\nnt\nno$. By definition of $\SLL$ again, we have
$\br_\brdip\SLLe\aa\nnt\nno$. Using the previous argument repeatedly, we
obtain $\br_\kk=\aa\nnt\nno$ for
$\kk\ge3$. Using the argument once more gives $\br_2\SLLe1$,
which implies $\br_2=1$. Finally, by definition of $\SLL$, we have $\br_1=1$.
We conclude by $(i)$ and the uniqueness of the $\ff\nn$-splitting.
\end{proof}

Owing to Proposition~\ref{P:Sep}, it is coherent to extend Definition~\ref{D:Sep} by $\sep\nn0 =\nobreak \aa\nno\nn$. In this way, the result of
Proposition~\ref{P:Sep}$(iii)$ extends to the case~$\brdi=0$.

\setlength{\unitlength}{1pt}
\begin{figure}[htb]
\begin{center}
\begin{picture}(301,56)
\put(0,0){\includegraphics[scale=0.6]{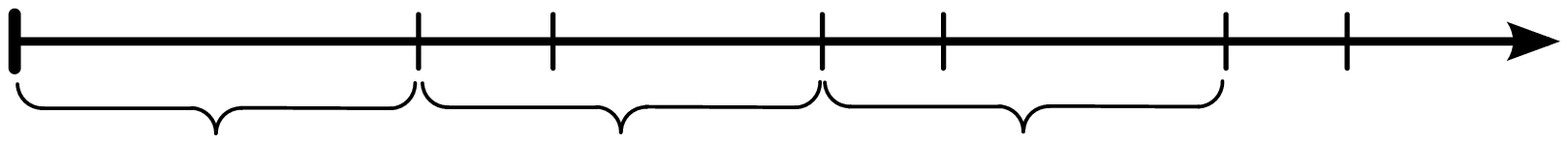}}
\put(12,45){$1$}
\put(80,45){$\sep\nn1$}
\put(108,45){$\ddd\nn$}
\put(152,45){$\sep\nn2$}
\put(179,45){$\ddd\nn^2$}
\put(224,45){$\sep\nn3$}
\put(251,45){$\ddd\nn^3$}
\put(22,8){\footnotesize $\nn$-breadth $\le2$}
\put(103,8){\footnotesize  $\nn$-breadth $3$}
\put(176,8){\footnotesize $\nn$-breadth $4$}
\end{picture}
\end{center}
\caption{\sf The braid $\sep\nn\rr$ as a separator in $(\BKL\nn,\SLL)$---hence in
$(\BKL\nn,<)$ as well once Theorem~\ref{T:coinOrdre} is proved.}
\label{F:Sep}
\end{figure}

We now observe that the orderings $<$ and $\SLL$ agree on the
separators~$\sep\nn\brdi$.

\begin{lemm}
\label{L:OrderSep}
Assume $\nn\ge3$. Then $0\le\brdi<\brdii$ implies $\sep\nn\brdi<\sep\nn\brdii$.
\end{lemm}

\begin{proof}
Assume $0<\brdi<\brdii$.
By Proposition~\ref{P:Sep}$(ii)$, we have 
\[
\sep\nn\brdi\inv\cdot\sep\nn{\brdii}=\ddd\nno^\brdi\,\ddd\nn^{-\brdi}\cdot\ddd\nn^{\brdii}\,\ddd\nno^{-\brdii}=\ddd\nno^\brdi\,\ddd\nn^{\brdii-\brdi}\,\ddd\nno^{-\brdii}.
\]
The hypothesis $\brdii-\brdi>0$ implies that 
$\sep\nn\brdi\inv\,\sep\nn\brdii$ is a
$\sig\nno$-positive braid, since the braid $\ddd\kk$ is $\sig\kko$-positive. Hence
we have~$\sep\nn\brdi<\sep\nn\brdii$.

It remains to establish the result for $\brdi=0$.
The previous case implies $\sep\nn1\le\sep\nn\brdii$.
As the relation~$<$ is transitive (Lemma~\ref{L:Transitive}), it is
enough to prove $\sep\nn0<\sep\nn1$. Using Proposition~\ref{P:Sep}$(ii)$ and inserting
$\ddd\nn\ddd\nn\inv$ on the left, we obtain
\[
\sep\nn0\inv\,\sep\nn1=\aa\nno\nn\inv\,\ddd\nn\,\ddd\nno\inv=\ddd\nn\,\ddd\nn\inv\aa\nno\nn\inv\ddd\nn\,\ddd\nno\inv.
\]
Relation~\eqref{E:GarsideAutoOnLetter} implies
$\ddd\nn\inv\aa\nno\nn\inv\ddd\nn=\ff\nn\inv(\aa\nno\nn\inv)=\aa\nnt\nno\inv$.
We deduce 
\[\sep\nn0\inv\,\sep\nn1 = \ddd\nn\,\aa\nnt\nno\inv\,\ddd\nno\inv,\]
and the latter decomposition is explicitly $\sig\nno$-positive.
\end{proof}

\subsection{The main result}
\label{SS:TheMainResult}

At this point, we have two \emph{a priori} unrelated linear orderings of the
monoid~$\BKL\nn$, namely the standard braid ordering~$<$, and the rotating
ordering~$\SLL$ of Definition~\ref{D:SLn}.  The main technical result of this paper is:

\begin{thrm}
\label{T:coinOrdre}
For all braids $\br, \brr$ in~$\BKL\nn$, the relation $\br\SLL\brr$ implies $\br < \brr$.
\end{thrm}

Before starting the proof of this result, we list a few consequences.
First we obtain a new proof of Property~$\textbf{C}$.

\begin{coro} 
[Property~$\textbf{C}$]
Every non-trivial braid is $\sigg$-positive or $\sigg$-negative.
\end{coro}

\begin{proof}
Assume that $\br$ is a non-trivial braid of $\BB\nn$.
First, as $\BB\nn$ is a group of fractions for~$\BKL\nn$, there exist $\brr, \brrr$
in~$\BKL\nn$ satisfying $\br = \brr{}\inv\,\brrr$. As $\br$ is assumed to be nontrivial,
we have $\brr \not= \brrr$. As $\SLL$ is a strict linear ordering, one of $\brr
\SLL \brrr$ or $\brrr \SLL \brr$ holds.  In the first case, Theorem~\ref{T:coinOrdre} implies that $\brr{}\inv\,\brrr$, \ie, $\br$, is
$\sigg$-positive. In the second case, Theorem~\ref{T:coinOrdre} implies that $\brrr{}\inv\,\brr$ is 
$\sigg$-positive, hence $\br$ is $\sigg$-negative.
\end{proof}

\begin{coro}
\label{C:Coincide}
The relation~$\SLL$ coincide with the restriction of~$<$ to~$\BKL\nn$.
\end{coro}

\begin{proof}
Let $\br, \brbr$ belong to~$\BKL\nn$. By Theorem~\ref{T:coinOrdre}, $\br\SLL\brbr$ implies $\br<\brbr$. Conversely, assume $\br
\not\SLL \brbr$. As $\SLL$ is a linear ordering, we have either $\brbr \SLL \br$, hence $\brbr < \br$, or $\br = \brbr$. In both cases,
Property~$\textbf{A}$ implies that $\br < \brbr$ fails. 
\end{proof}

Corollary~\ref{C:Coincide} directly implies Theorem~1 stated in the introduction.
Indeed, the characterization of the braid ordering given in Theorem~1 is nothing but
the recursive definition of the ordering~$\SLL$.

Finally, we obtain a new proof of Laver's result, together with a determination of the order type.

\begin{coro}
\label{C:OrderTypeDeh}
The restriction of the Dehornoy ordering to the dual braid mon\-oid~$\BKL\nn$ is a
well-ordering, and its order type is the ordinal $\omega^{\omega^\nnt}$.
\end{coro}

\begin{proof}
It is standard that, if $(X, <)$ is a well-ordering of ordinal type~$\lambda$, then the 
$\ShortLex$-extension of~$<$ to the set of all finite sequences of elements of~$X$ is a
well-ordering of ordinal type~$\lambda^\omega$---see \cite{Levy}. The ordinal type
of~$\SLL$ on~$\BKL2$ is $\omega$, the order type of the standard ordering of natural
numbers. So, an immediate induction shows that, for each~$\nn \ge 2$, the ordinal type
of~$\SLL$ on~$\BKL\nn$ is at most~$\omega^{\omega^\nnt}$. 

\emph{A priori}, this is only an upper bound, because it is not true that every
sequence of braids in~$\BKL\nno$ is the $\ff\nn$-splitting of a braid of~$\BKL\nn$.
However, by construction, the monoid~$\BKL\nn$ includes the positive braid
monoid~$\BP\nn$, and it was shown in~\cite{Burckel:WO}---or, alternatively,
in~\cite{Carlucci}---that the order type of the restriction of the braid ordering
to~$\BP\nn$ is
$\omega^{\omega^\nnt}$. Hence the ordinal type of its restriction to~$\BKL\nn$ is at
least that ordinal, and, finally, we have equality. (Alternatively, we could also directly
construct a type $\omega^{\omega^\nnt}$ increasing sequence in~$\BKL\nn$.)
\end{proof}

\begin{rema}
By construction, the ordering~$<$ is invariant under left-multiplica\-tion. 
Another consequence of Corollary~\ref{C:Coincide} is that the
ordering~$\SLL$ is invariant under left-multiplication as well. Note that the latter result is \emph{not} obvious at all
from the direct definition of that relation.
\end{rema}

\section{A Key Lemma}
\label{S:KeyLemma}

So, our goal is to prove that the rotating ordering of Definition~\ref{D:SLn} and the
standard braid ordering coincide. The result will follow from the fine properties of the
rotating normal form and of the $\ff\nn$-splitting. The aim of this section is to 
establish these properties. Most of them are improvements of properties
established in~\cite{F:SR}, and we shall heavily use the notions introduced in this
paper.

\subsection{Sigma-positive braid of type $\aa\indi\nn$.}
\label{SS:SigmaPositiveBraidOfType}

We shall prove Theorem~\ref{T:coinOrdre} by using an induction on the number of 
strands~$\nn$. Actually, in order to maintain an induction hypothesis, we shall prove a
stronger implication: instead of merely proving that, if $\br$ is $\SLL$-smaller
than~$\brbr$, then the quotient braid~$\br\inv
\brbr$ is $\sigg$-positive, we shall prove the more precise conclusion that $\br\inv
\brbr$ is
\emph{$\sigg$-positive of type
$\aa\indi\nn$} for some~$\indi$ related to the last letter in~$\brbr$.

\begin{defi} 
\label{D:Dangerous}
Assume $\nn\ge3$.
\begin{deflist}
 \item A braid is called \emph{$\aa\indi\nn$-dangerous} if it admits one
decomposition of the form
\[\dddd{\findi\dd}\nno\inv\ \dddd{\findi\ddo}\nno\inv\ \Ldots\ \dddd{\findi1}\nno\inv,\]
with $\findi\dd\ge\findi\ddo\ge\Ldots\ge\findi1=\indi$.
\item A braid is called \emph{$\sig\ii$-nonnegative} if it is $\sig\ii$-positive or it belongs
to~$\BB\ii$.
\item  For $\indi\le\nnt$, a braid~$\br$ is called \emph{$\sigg$-positive of type 
$\aa\indi\nn$} if it can be expressed as \[\brpos\cdot\dddd\indi\nn\cdot\brneg,\] where
$\brpos$ is $\sig\nno$-nonnegative and $\brneg$ is $\aa\indi\nn$-dangerous.
\item  A braid~$\br$ is called \emph{$\sigg$-positive of type $\aa\nno\nn$} if it is $1$, or equal to \[\brr\cdot\aa\nno\nn,\]
 where $\brr$ is a $\sigg$-positive braid of type $\aa1\nn$.
\end{deflist}
\end{defi}

Note that, an $\aa\indi\nn$-braid with $\indi\not=\nno$ is not the trivial one, \ie, is different from~$1$, as it contains $\dddd\indi\nno$, which is non trivial.
In the other hand, the only $\aa\nno\nn$-danger ours braid is $1$.

We observe that the definition of~$\sigg$-positive of type~$\aa\nno\nn$ is  different from the
definition of $\sigg$-positive of type~$\aa\indi\nn$ for $\indi < \nno$ (technical reasons make such a
distinction necessary).

Saying a braid is $\sigg$-positive of type $\aa\indi\nn$ is motivated by the fact that $\aa\indi\nn$ is the simplest $\sigg$-positive braid of its type.
 
\begin{lemm}
\label{L:TypeSigmaPos}
Assume that $\br$ is a $\sigg$-positive braid of type $\aa\indi\nn$. Then

$(i)$ $\br$ is $\sig\nno$-positive,

$(ii)$ $\ff\nnp(\br)$ is $\sigg$-positive of type $\aa\indip\nnp$,

$(iii)$ if $\indi=1$ then $\br\,\ddd\nno^{-\tt}$ is $\sigg$-positive of type
$\aa1\nn$ for all $\tt\ge0$,

$(iv)$ if $\br\not=\aa\nno\nn$ holds, then $\brbr\,\br$ is $\sigg$-positive of type~$\aa\indi\nn$ for every $\sig\nno$-nonnegative braid $\brbr$.
\end{lemm}

\begin{proof}
$(i)$ An $\aa\indi\nn$-dangerous braid is $\sig\nno$-nonnegative (actually it is
$\sig\nnt$-negative), and the braid $\dddd\indi\nn$ is $\sig\nno$-positive. Therefore
$\br$ is $\sig\nno$-positive.

$(ii)$ With the notation of Definition~\ref{D:Dangerous}, let
$\dddd{\findi\dd}\nno\inv\
\Ldots\ \dddd{\findi1}\nno\inv$ be
the decomposition of $\brneg$, with $\findi1=\indi$. Then we have
$$\ff\nnp(\brneg) = \dddd{\findi\dd+1}\nn\inv\ \Ldots\ \dddd{\findi1+1}\nn\inv,$$
an $\aa\indip\nnp$-dangerous word. By definition, the braid $\brpos$ can be represented
by a word on the alphabet $\sigpm\ii$ with $\ii\le\nnt$. As, for $\ii\le\nnt$, the
image of $\sig\ii$ by $\ff\nnp$ is $\sig\iip$, the braid $\ff\nnp(\brpos)$ is
$\sig\nn$-nonnegative. So the relation
$\ff\nnp(\br)=\ff\nnp(\brpos)\cdot\dddd\indip\nnp\cdot\ff\nnp(\brneg)$
witnesses that $\ff\nnp(\br)$ is $\sigg$-positive of type $\aa\indip\nnp$.

Point~$(iii)$ directly follows from the fact that, if $\brbrneg$ is an
$\aa1\nn$-dangerous braid, then, for each $\tt\ge0$, the braid
$\brbrneg\,\ddd\nno^{-\tt}$ is also
$\aa1\nn$-dangerous.

$(iv)$ Assume $\indi\le\nnt$.
Then, by definition, we have $\br=\brpos\cdot\dddd\indi\nn\cdot\brneg$, where $\brpos$ is $\sig\nno$-nonnegative and $\brneg$ is $\aa\indi\nn$-dangerous.
Hence, we get $\brbr\,\br=\brbr\,\brpos\cdot\dddd\indi\nn\cdot\brneg$.
As the product of $\sig\nno$-nonnegative braids is $\sig\nno$-nonnegative, the braid $\brbr\,\br$ is $\sigg$-positive of type $\aa\indi\nn$.

Assume now $\indi=\nno$.
As, by hypothesis $\br$ is different from $\aa\nno\nn$, we have $\br=\brr\cdot\aa\nno\nn$, where $\brr$ is $\sigg$-positive of type $\aa1\nn$.
The case $\indi\le\nnt$ implies that the braid $\brbr\,\brr$ is $\sigg$-positive of type $\aa1\nn$.
Hence the braid $\brbr\,\br$, which is equal to $\brbr\,\brr\cdot\aa\nno\nn$, is
$\sigg$-positive of type~$\aa\nno\nn$.
\end{proof}

\begin{rema}
For $\tt\ge1$, the braid $\sep\nn\tt$ is $\sigg$-positive of type $\aa1\nn$.
Indeed, by Proposition~\ref{P:Sep}$(ii)$, we have $\sep\nn\tt=\ddd\nn^\tto\cdot\ddd\nn\cdot\ddd\nno^{-\tt}$, the right-hand side being an explicit $\sigg$-positive braid of type $\aa1\nn$.
\end{rema}

\subsection{Properties of $\sigg$-positive braids of type $\aa\indi\nn$}
\label{SS:Properties}

We now show that the entries in a $\ff\nn$-splitting
give raise to $\sigg$-positive braids of type~$\aa\indi\nn$, for some~$\indi$ that can
be effectively controlled.

\begin{lemm}
\label{L:SigsPositive}
For $\nn\ge3$, every braid with last letter~$\aa\indi\nn$ is $\sigg$-positive of type
$\aa\indi\nn$.
\end{lemm}

\begin{proof}
Let $\br$ be a braid of $\BKL\nn$ with last letter~$\aa\indi\nn$
(Definition~\ref{D:Last}). Put $\br = \brr\cdot\aa\indi\nn$.
Assume first $\indi\le\nnt$.
Then, by~\eqref{E:AConjugateD}, we have $\br
= \brr\cdot\dddd\indi\nn\cdot\dddd\indi\nno\inv$, an explicit $\sigg$-positive braid of
type $\aa\indi\nn$, since the braid $\brr$ is positive, hence $\sig\nno$-nonnegative, and $\dddd\indi\nno\inv$ is $\aa\indi\nn$-dangerous.

Assume now $\indi=\nno$.
The case $\brr=1$ is clear.
For $\brr\not=1$,  by Lemma~\ref{L:LastLetter}(iii), there is a positive braid $\brrr$ satisfying
$\brr=\brrr\cdot\aa\indii\nn$ for some $\indii$. The relation
$\aa1\indii\,\aa\indii\nn=\aa\indii\nn\,\aa1\nn$ implies
$\aa\indii\nn=\aa1\indii\inv\,\aa\indii\nn\,\aa1\nn$. Using~\eqref{E:AConjugateD} for
$\aa1\nn$ gives
\[
 \brr=\brrr\,\aa1\indii\inv\,\aa\indii\nn\cdot\dddd1\nn\cdot\dddd1\nno\inv,
\]
an explicit $\sigg$-positive braid of type~$\aa1\nn$.
Therefore, $\br$ is $\sigg$-positive of type~$\aa1\nn$.
\end{proof}

We shall see now that the normal form of every braid $\br$ of $\BKL\nno$ such that the~$\BKL\nno$-tail of $\ff\nn(\aa\indi\nn\,\br)$ is trivial contains a sequence of overlapping letters $\aa\indiii\indiv$. Words containing such sequences are what we shall call ladders.

\begin{defi}
\label{D:Ladder}
For $\nn\ge 3$, we say that a normal word~$\ww$ is an
\emph{$\aa\indi\nn$-ladder lent on~$\aa\indiio\nno$}, if
there exists a decomposition
\begin{equation}
\label{E:Ladder}
\ww = \ww_0\,\xx_1\,\ww_1 \,...\, \ww_\hho\,\xx_\hh\,\ww_\hh,
\end{equation}
and a sequence $\indi = \findi0 < \findi1 < ... < \findi\hh = \nno$ such that

$(i)$ for each~$\kk\le\hh$, the letter~$\xx_\kk$ is of the form  $\aa{\find\kk}{\findi\kk}$ with $\find\kk{<}\findi{\kko}{<}\findi\kk$,

$(ii)$ for each~$\kk<\hh$, the word~$\ww_\kk$ contains no letter $\aa\indi\indii$ with $\indi{<}\findi\kk{<}\indii$,

$(iii)$ the last letter of $\ww$ is~$\aa\indiio\nno$.
\end{defi}

By convention, an $\aa\nno\nn$-ladder lent on~$\aa\indiio\nno$ is a word on the letters $\aa\indi\indii$ whose last letter is $\aa\indiio\nno$.

The concept of a ladder is easily illustrated by representing the
generators~$\aa\indi\indii$ as a vertical line from the~$\indi$th line to
the~$\indii$th line on an $\nn$-line stave.  Then, for every $\kk\ge0$, the
letter~$\xx_\kk$ looks like a bar of a ladder---see Figure~\ref{F:Ladder}.

\setlength{\unitlength}{1mm}
\begin{figure}[htb]
\begin{picture}(110,23)
\put(2,0){\includegraphics{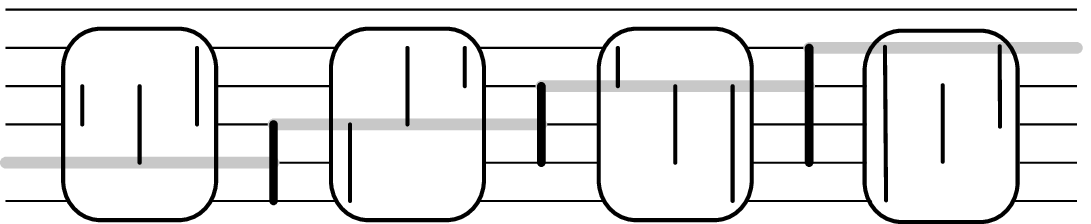}}
\put(0,1){\footnotesize $1$}
\put(0,21){\footnotesize $6$}
\end{picture}
\caption{{\sf \smaller The bars of the ladder are represented by black thick vertical lines.
 An $\aa25$-ladder lent on~$\aa35$ (the last letter).
The gray line starts at position $2$ and goes up to position $5$ using the bars of the ladder.
The empty spaces between bars in the ladder are represented by a framed box.
In such boxes the vertical line representing the letter~$\aa\indi\indii$ does not cross the gray line.
}}
\label{F:Ladder}
\end{figure}

\begin{prop}\cite{F:SR}
\label{P:SplittingAndLadders}
Assume $\nn\ge3$ and that $(\br_\brdi,\Ldots, \br_1)$ is the $\ff\nn$-splitting of some braid of~$\BKL\nn$.
Then, for each~$\kk$ in~$\{\brdio, \Ldots, 3\}$, the normal form
of~$\br_\kk$ is a  $\ff\nn(\last{\br}_\kkp)$-ladder lent on~$\last{\br}_\kk$.
The same results hold for $\kk=2$ whenever $\br_2$ is not $1$.
\end{prop}

A direct consequence of Proposition 5.7 of~\cite{F:SR} is

\begin{lemm}
\label{L:DangerousAgainstLadder}
Assume $\nn\ge3$ and let $\br$ be a braid represented by an  $\aa\indi\nn$-ladder lent on~$\aa\indiio\nno$ with $\indii\not=\nno$ and $\brbrneg$ be an $\aa\indi\nn$-dangerous braid.
Then $\brbrneg\,\br$ is a $\sigg$-positive braid of type~$\aa\indiio\nno$.
\end{lemm}

We are now ready to prove that the non-terminal entries of  a $\ff\nn$-splitting $(\br_\brdi,\Ldots,\br_1)$ have the expected property, namely that the braid $\br_\kk$ provides a protection against a $\ff\nn(\last{\br}_\kkp)$-dangerous braid, in the sense that if $\brbrneg_\kkp$ is a $\last{\br}_\kkp$-dangerous braid, then the braid $\ff\nn(\brbrneg_\kkp)\,\br_\kk$ is $\sigg$-positive of type $\last{\br}_\kk$.

\begin{prop}
\label{P:DangerousAgainstStair} 
Assume that $(\br_\brdi,\Ldots,\br_1)$ is a $\ff\nn$-splitting. Then, for each $\kk$ in $\{\brdio,\Ldots,3\}$ and every
$\last{\br}_\kkp$-dangerous braid $\brbrneg_\kkp$, the braid
$\ff\nn(\brbrneg_\kkp)\,\br_\kk$ is $\sigg$-positive of type~$\last{\br}_\kk$. 
Moreover  $\brbrneg_\kkp\,\br_\kk$ is different from $\aa\nnt\nno$, except if $\br$ is itself~$\aa\nnt\nno$. The same result holds for $\kk=2$, unless $\br_2$ is the trivial braid$1$.
\end{prop}

\begin{proof}
Take $\kk$ in $\{\brdio,\Ldots,3\}$.
By definition of a dangerous braid, the braid $\ff\nn(\brbrneg_\kkp)$ is $\ff\nn(\last{\br}_\kkp)$-dangerous.
Assume $\last{\br}_\kk\not=\aa\nnt\nno$.
By Proposition~\ref{P:SplittingAndLadders}, the normal form of $\br_\kk$ is a
$\ff\nn(\last{\br}_\kkp)$-ladder lent on $\last{\br}_\kk$. Then, by Lemma~\ref{L:DangerousAgainstLadder}, 
the braid $\ff\nn(\brbrneg_\kkp)\,\br_\kk$ is $\sigg$-positive of type~$\last{\br}_\kk$.

Assume now $\last{\br}_\kk=\aa\nnt\nno$ with $\br_\kk\not=\aa\nnt\nno$.
By Proposition~\ref{P:SplittingAndLadders} again, the normal form of $\br_\kk$ is a $\ff\nn(\last{\br}_\kkp)$-ladder lent on $\aa\nnt\nno$.
Let $\www\,\aa\nnt\nno$ be the normal form of $\br$.
By definition of a ladder, as the letter $\aa\nnt\nno$ does not satisfy the condition $(i)$ of Definition~\ref{D:Ladder}, the word  $\www$ is an $\aa\indi\nn$-ladder lent on
$\aa{\indi}\nno$ for some $\indi$---see Lemma~\ref{L:LastLetter}$(iii)$.
We denote by $\brr_\kk$ the braid represented by $\www$.
Then, by Lemma~\ref{L:DangerousAgainstLadder}, the braid $\ff\nn(\brbrneg_\kkp)\,\brr_\kk$ is $\sigg$-positive of type $\aa\indi\nno$.
Then it is the product $\brrpos_\kk\cdot\dddd\indi\nn\cdot\,\brrneg_\kk$.
The relation $\dddd1\indi\,\dddd\indi\nno=\dddd1\nno$ implies that the braid $\ff\nn(\brbrneg_\kkp)\,\brr_\kk$ is equal to
\[
 \brrpos_\kk\,\dddd1\indi\inv\cdot\dddd1\nno\cdot\,\brrneg_\kk,
\]
where $\brrpos_\kk\,\dddd1\indi\inv$ is $\sig\nnt$-nonnegative and $\brrneg_\kk$ is $\aa1\nn$-dangerous.
Then $\ff\nn(\brbrneg_\kkp)\,\brr_\kk$ is $\sigg$-positive of type $\aa1\nno$.
Hence $\ff\nn(\brbrneg_\kkp)\,\br_\kk$ is $\sigg$-positive of type $\aa\nnt\nno$.

Assume finally $\br_\kk=\aa\nnt\nno$.
As the only $\aa\nnt\nno$-dangerous braid is trivial, the braid $\ff\nn(\brbrneg_\kkp)\,\br_\kk$ is
equal to $\aa\nnt\nno$, a $\sigg$-positive braid of type $\aa\nnt\nno$.

The same arguments establish the case $\kk=2$ with $\br_2\not=1$.
\end{proof}

\subsection{The Key Lemma}

We arrive at our main technical result. It mainly says that, if a braid~$\br$
of~$\BKL\nn$ has $\nn$-breadth~$\brdi$, then the braid $\sep\nn\brdit\inv\cdot
\br$ is either $\sigg$-positive or trivial. Actually, the result is stronger: the additional information is first that we can control the
type of the quotient above, and second that a similar result holds when we replace the
leftmost entry of the $\ff\nn$-splitting of~$\br$ with another braid of~$\BKL\nno$ that
resembles it enough. This stronger result, which unfortunately makes the statement more
complicated, will be needed in Section~\ref{S:ProofOfTheMainResult} for the final
induction on the braid index~$\nn$.

\begin{prop}
\label{P:KeyLemma}
Assume $\nn\ge3$ and that $(\br_\brdi,\Ldots,\br_1)$ is the $\ff\nn$-splitting of a braid~$\br$ in~$\BKL\nn$ with $\brdi\ge3$.
Let $\aa\indii\nn$ be the last letter of $\br\,\br_1\inv$.
Whenever $\brbr_\brdi$ is a  $\sigg$-positive braid of type $\last{\br}_\brdi$, the braid
\begin{equation}
\label{E:KeyLemma:Enonce}
\sep\nn\brdit\inv\cdot\ff\nn^{\brdio}(\brbr_\brdi)\cdot\ff\nn^{\brdit}(\br_\brdio)\cdot...\cdot\ff\nn(\br_2),
\end{equation}
is trivial or $\sigg$-positive of type $\aa\indii\nn$---the first case occurring only for
$\indii=1$.
\end{prop}

\begin{proof}
Put $\br^*=\sep\nn\brdit\inv\cdot\ff\nn^{\brdio}(\brbr_\brdi)\cdot\ff\nn^{\brdit}(\br_\brdio)
\cdot...\cdot\ff\nn(\br_2)$ and $\aa\indio\nno = \last{\br}_\brdi$ (Lemma~\ref{L:LastLetter}). First, we decompose  the left fragment
$\sep\nn\brdit\inv\cdot\ff\nn^{\brdio}(\brbr_\brdi)$ of~$\br^*$ as a product of a
$\sig\nno$-nonnegative braid and a dangerous braid. By definition of a
$\sigg$-positive braid of type $\last{\br}_\brdi$, we have 
\[
\brbr_\brdi=\brbrpos_\brdi\,\dddd\indio\nno\,\brbrneg_\brdi, 
\]
where $\brbrneg_\brdi$ is an $\last{\br}_\brdi$-dangerous braid and where $\brbrpos_\brdi$ is $\sig\nnt$-nonnegative.
Using Proposition~\ref{P:Sep}$(ii)$t, we obtain
\begin{equation}
 \label{E:KeyLemma:2}
\sep\nn\brdit\inv\cdot\ff\nn^{\brdio}(\brbr_\brdi)=\ddd\nno^{\brdi-2}\,\ddd\nn^{-\brdi+2}\ff\nn^{\brdio}(\brbrpos_\brdi\,\dddd\indio\nno)\,\ff\nn^\brdio(\brbrneg_\brdi).
\end{equation}
By~\eqref{E:PushDdd}, we have
$\ddd\nn^{-\brdi+2}\ff\nn^{\brdio}(\brbrpos_\brdi\,\dddd\indio\nno)
= \ff\nn(\brbrpos_\brdi\,\dddd\indio\nno)\ddd\nn^{-\brdi+2}$.
Using the relation
$\dddd\indi\nn\,\ddd\nn\inv=\ddd\indi\inv$, an easy
consequence of~\eqref{E:RelationD:ii}, we obtain
\begin{equation}
\label{E:KeyLemma:3}
\ddd\nn^{-\brdi+2}\ff\nn^{\brdio}(\brbrpos_\brdi\,\dddd\indio\nno)=\ff\nn(\brbrpos_\brdi)\,\ddd\indi\inv\,\ddd\nn^{-\brdi+3}.
\end{equation}
Substituting \eqref{E:KeyLemma:3} in \eqref{E:KeyLemma:2}, we find
\begin{equation}
 \label{E:KeyLemma:4}
\sep\nn\brdit\inv\cdot\ff\nn^{\brdio}(\brbr_\brdi)=\ddd\nno^{\brdi-2}\,\ff\nn(\brbrpos_\brdi)\,\ddd\indi\inv\,\ddd\nn^{-\brdi+3}\,\ff\nn^\brdio(\brbrneg_\brdi).
\end{equation}
From there, we deduce that $\br^*$ is equal to
\begin{equation}
 \label{E:KeyLemma:5}
\ddd\nno^{\brdi-2}\,\ff\nn(\brbrpos_\brdi)\,\ddd\indi\inv\,\cdot\,\ddd\nn^{-\brdi+3}\,\ff\nn^\brdio(\brbrneg_\brdi)\,\ff\nn^\brdit(\br_\brdio)\,\Ldots\ff\nn(\br_2).
\end{equation}
Write $\br^{**}=\ddd\nn^{-\brdi+3}\,\ff\nn^\brdio(\brbrneg_\brdi)\,\ff\nn^\brdit(\br_\brdio)\,\Ldots\ff\nn(\br_2)$.
Note that  the left factor of \eqref{E:KeyLemma:5}, which is $\ddd\nno^{\brdi-2}\,\ff\nn(\brbrpos_\brdi)\,\ddd\indi\inv$, is $\sig\nno$-nonnegative.
At this point, four cases may occur.

\smallskip
\noindent{\bf Case 1:} $\br_2 \notin \{1, \aa\nnt\nno\}$.
By Lemma 5.9 of \cite{F:SR}, the braid $\br^{**}$ is equal to~$\brrr\,\ff\nn^2(\brbrneg_3)\,\ff\nn(\br_2)\,\br_1$, where $\brrr$ is a
$\sig\nno$-nonnegative braid and where $\brbrneg_3$ is a $\last{\br}_3$-dangerous braid. Put
$\brr_2 = \ff\nn(\brbrneg_3)\,\br_2$. By
Proposition~\ref{P:DangerousAgainstStair}, $\brr_2$ is $\sigg$-positive of type
$\last{\br}_2$ and different from $\aa\nnt\nno$. We deduce that $\br^*$ is equal to 
\begin{equation}
 \label{E:KeyLemma:6}
\ddd\nno^{\brdi-2}\,\ff\nn(\brbrr_\brdi)\,\ddd\indi\inv\,\brrr\,\cdot\,\ff\nn(\brr_2).
\end{equation}
The left factor of \eqref{E:KeyLemma:6} is $\sig\nno$-nonnegative, while the right factor, namely $\ff\nn(\brr_3)$, is different from~$\aa\nno\nn$ and $\sigg$-positive of type $\ff\nn(\last{\br}_2)$ by Lemma~\ref{L:TypeSigmaPos}$(ii)$.
As, in this case,  the last letter of $\br\,\br_1\inv$ is $\ff\nn(\last{\br}_2)$, we conclude using Lemma~\ref{L:TypeSigmaPos}$(iv)$.

\smallskip
\noindent{\bf Case 2:} $\br_2 \in \{1, \aa\nnt\nno\}$, $\br_3=\Ldots=\br_\kko=\aa\nnt\nno$ and $\br_\kk\not=\aa\nnt\nno$ for some $\kk\le\brdio$. If $\br_2$ is trivial, then the last letter of $\br\,\br_1\inv$ is $\aa1\nn$;  otherwise the last
letter of
$\br\,\br_1\inv$ is $\ff\nn(\aa\nnt\nno)$, \ie, $\aa\nno\nn$---a direct consequence of Lemma~\ref{L:LastLetter}. 
As the product of a $\sigg$-positive braid of type $\aa1\nn$ with $\aa\nno\nn$ is a $\sigg$-positive braid of type
$\aa\nno\nn$, it is enough to prove that the braid $\br^*$ is the product of a $\sigg$-positive braid of type~$\aa1\nn$ with $\ff\nn(\br_2)$.
By Lemma 5.9 of \cite{F:SR}, the braid $\br^**$ is equal to
\begin{equation}
 \label{E:KeyLemma:8}
 \brrr\,\ddd\nn^{-\kk+2}\ff\nn^\kk(\brbrneg_\kkp)\,\ff\nn^\kko(\br_\kk)\,\ff\nn^\kkt(\aa\nnt\nno)\,\Ldots\,\ff\nn^2(\aa\nnt\nno)\,\ff\nn(\br_2),
\end{equation}
with $\brrr$ a $\sig\nno$-nonnegative braid and $\brbrneg_\kkp$ a $\last{\br}_\kkp$-dangerous braid.
Proposition~\ref{P:DangerousAgainstStair} implies that the braid $\ff\nn(\brbrneg_\kkp)\,\br_\kk$ is $\sigg$-positive of type $\last{\br}_\kk$.
By Corollary 3.11 of~\cite{F:SR}, the last letter of $\br_\kk$ is $\aa\nnt\nno$.
Then, by Lemma~\ref{L:TypeSigmaPos}$(ii)$ the braid $\ff\nn^2(\brbr_\kkp)\,\ff\nn(\br_\kk)$ is $\sigg$-positive of type $\aa\nno\nn$.
Hence, by definition of a $\sigg$-positive braid of type~$\aa\nno\nn$, we have the relation
\begin{equation}
 \label{E:KeyLemma:9}
\ff\nn^2(\brbrneg_\kkp)\,\ff\nn(\br_\kk)=\brr_\kk\,\ff\nn(\aa\nnt\nno),
\end{equation}
where $\brr_\kk$ is a $\sigg$-positive braid of type $\aa1\nn$.
Substituting \eqref{E:KeyLemma:9} in \eqref{E:KeyLemma:8} gives that~$\br^{**}$ is equal to 
\[
\brrr\,\ddd\nn^{-\kk+2}\ff\nn^\kkt(\brr_\kk)\,\ff\nn^\kko(\aa\nnt\nno)\,\ff\nn^\kkt(\aa\nnt\nno)\,\Ldots\,\ff\nn^2(\aa\nnt\nno)\,\ff\nn(\br_2).
\]
Using $\ff\nn(\aa\nnt\nno)\,\ddd\nn\inv=\ddd\nno\inv$ and \eqref{E:PushDdd}, we obtain that the
right factor of \eqref{E:KeyLemma:5} is
\[
\label{E:KeyLemma:11}
\brrr\,\brr_\kk\,\ddd\nno^{-\kk+2}\,\ff\nn(\br_2).
\]
As $\brr_\kk$ is a $\sigg$-positive braid of type $\aa1\nn$, Lemma~\ref{L:TypeSigmaPos}$(iii)$ implies that $\brr_\kk\,\ddd\nno^{-\kk+2}$ is $\sigg$-positive of type $\aa1\nn$, and so is $\brrr\brr_\kk\,\ddd\nno^{-\kk+2}$ by Lemma~\ref{L:TypeSigmaPos}$(iv)$.
Hence, by~\eqref{E:KeyLemma:5}, the braid $\br^*$ is the product of a $\sigg$-positive braid of type $\aa1\nn$ with $\ff\nn(\br_2)$.

\smallskip
\noindent{\bf Case 3:} $\br_2 \in \{1, \aa\nnt\nno\}$, $\br_3 =\Ldots=\br_\brdio=\aa\nnt\nno$ and $\brbr_\brdi \not= \aa\nnt\nno$. As in Case~$2$, it is enough to prove that the
braid~$\br^*$ is the product  of a $\sigg$-positive of  type~$\aa1\nn$ with $\ff\nn(\br_2)$.
Using Lemma~\ref{P:Sep}$(ii)$ and \eqref{E:PushDdd} in its definition, the braid~$\br^*$ is equal to
\[
\ddd\nno^{\brdi-2}\cdot\ff\nn(\brbr_\brdi)\,\ddd\nn\inv\cdot\ff\nn(\aa\nnt\nno)\,\ddd\nn\inv\cdot\,\Ldots\cdot\ff\nn(\aa\nnt\nno)\,\ddd\nn\inv\,\ff\nn(\br_2).
\]By Corollary 3.11 of \cite{F:SR}, the last letter of $\br_\brdi$ is $\aa\nnt\nno$, so 
$\brbr_\brdi$ is $\sigg$-positive of type~$\aa\nnt\nno$. Hence  $\ff\nn(\brbr_\brdi)$
is $\sigg$-positive of type $\aa\nno\nn$ and is different from~$\aa\nno\nn$. Then, by definition of a  $\sigg$-positive braid of type $\aa\nno\nn$, there exists a 
$\sigg$-positive braid $\brr_\brdi$ of type $\aa1\nn$ satisfying $\ff\nn(\brbr_\brdi)=\brr_\brdi\,\aa\nno\nn$.
Using $\ff\nn(\aa\nnt\nno)\,\ddd\nn\inv=\ddd\nno\inv$, we deduce that the braid $\br^*$ is equal to
\begin{equation}
 \label{E:FinCas3}
\ddd\nno^{\brdi-2}\cdot\brr_\brdi\ \ddd\nno^{-\brdi+2}\cdot\ff\nn(\br_2).
\end{equation}
By Lemma~\ref{L:TypeSigmaPos}$(iii)$, the middle factor of \eqref{E:FinCas3}, namely $\brr_\brdi\ \ddd\nno^{-\brdi+2}$, is $\sigg$-positive of type $\aa1\nn$.
Then  $\br^*$ is the product of a $\sigg$-positive braid of type~$\aa1\nn$ with~$\ff\nn(\br_2)$.

\smallskip
\noindent{\bf Case 4:} $\br_2 \in \{1, \aa\nnt\nno\}$, $\br_3 =\Ldots=\br_\brdio=\aa\nnt\nno$ and $\brbr_\brdi =\aa\nnt\nno$.
By definition, we have
\begin{equation}
\label{E:KeyLemma:15}
\br^*=\ddd\nno^{\brdi-2}\cdot\ff\nn(\aa\nnt\nno)\,\ddd\nn\inv\,\Ldots\,\ff\nn(\aa\nnt\nno)\,\ddd\nn\inv\,\ff\nn(\br_2).
\end{equation}
Using $\ff\nn(\aa\nnt\nno)\,\ddd\nn\inv=\ddd\nno\inv$ once again, we deduce
$\br^* = \ff\nn(\br_2)$. The  braid $\ff\nn(\br_2)$ is either trivial or equal to $\aa\nno\nn$, a
$\sigg$-positive word of type $\aa\nno\nn$, as expected.

So the proof of the Key Lemma is complete.
\end{proof}

\section{Proof of the main result}
\label{S:ProofOfTheMainResult}

We are now ready to prove Theorem~\ref{T:coinOrdre}, which state that if $\br, \brbr$ are braids of~$\BKL\nn$, then
\begin{equation}
  \label{E:Main}
  \text{$\br \SLL \brbr$ \quad implies \quad $\br < \brbr$,}
\end{equation}
where $\SLL$ refers to the ordering of Definition~\ref{D:SLn} and $<$ refers to the
Dehornoy ordering, \ie, $\br < \brbr$ means that the quotient-braid~$\br\inv \brbr$ is
$\sigg$-positive.

We shall split the argument into three steps.
The first step consists in replacing the initial problem that involves
two arbitrary braids~$\br, \brbr$ with two problems, each of which only
involves one braid. To do that, we use the
separators~$\sep\nn\tt$ of Definition~\ref{D:Sep}, and address the problem of comparing
one arbitrary braid with the special braids~$\sep\nn\tt$. We shall prove that
\begin{gather}
  \label{E:Main1}
  \text{$\br \SLL \sep\nn\tt$ \quad implies \quad $\br < \sep\nn\tt$}\\
  \label{E:Main2}
  \text{$\sep\nn\tt \SLLe \br$ \quad implies \quad $\sep\nn\tt\le\br$}
\end{gather}
So, essentially, we have three things to do: proving~\eqref{E:Main1},
proving~\eqref{E:Main2}, and showing how to deduce the general
implication~\eqref{E:Main}. 

\subsection{Proofs of~\eqref{E:Main1} and~\eqref{E:Main2}}

We begin with the implication~\eqref{E:Main1}. Actually, we shall prove a stronger
result, needed to maintain an inductive argument in the proof of
Theorem~\ref{T:coinOrdre}.

\begin{prop}
 \label{P:Upper}
For $\nn\ge3$, the implication~\eqref{E:Main1} is true.
Moreover for $\tt\ge1$, the relation $\br\SLL\sep\nn\tt$ implies that $\br\inv\,\sep\nn\tt$ is $\sigg$-positive of type $\aa1\nn$.
\end{prop}

\begin{proof}
Take $\br$ in $\BKL\nn$ and assume $\br\SLL\sep\nn\tt$ for some $\tt\ge0$.
Let $(\br_\brdi,\Ldots, \br_1)$ be the $\ff\nn$-splitting of~$\br$.
By Proposition~\ref{P:Sep}$(iii)$, we necessarily have $\brdi\le\ttp$.
If $\tt=0$ holds, then the braid $\br$ lies in $\BKL\nno$ and the quotient~$\br\inv\sep\nn0$, which is $\br\inv\,\aa\nno\nn$, is $\sigg$-positive.
If $\tt\ge1$ and $\brdi\le1$ hold, then the braid $\br\inv$ is $\sig\nno$-nonnegative, and as $\sep\nn\tt$ is $\sigg$-positive of type $\aa1\nn$, Lemma~\ref{L:TypeSigmaPos} implies that the quotient~$\br\inv\sep\nn\tt$ is $\sigg$-positive of type $\aa1\nn$.
Assume now $\tt\ge1$ and $\brdi\ge2$.
Then, by Proposition~\ref{P:Sep}$(ii)$, we find 
\begin{equation*}
\br\inv\sep\nn\tt = \br\inv\,\ddd\nn^\tt\,\ddd\nno^{-\tt}=
\br_1\inv \cdot \ff\nn(\br_2\inv)\ \Ldots\ \ff\nn^\brdio(\br_\brdi\inv) \cdot \ddd\nn^\tt \cdot\ddd\nno^{-\tt}.
\end{equation*} 
Using Relation \eqref{E:PushDdd}, we push $\brdio$~factors $\ddd\nn$ to the 
left and dispatch them between the factors~$\br_\kk\inv$:
\begin{align*}
\br\inv \sep\nn\tt &= \br_1\inv \, \ff\nn(\br_2\inv) \ \Ldots\ \ff\nn^\brdit(\br_\brdio\inv)\, \ff\nn^\brdio(\br_\brdi\inv)\, \ddd\nn^\brdio \ \ddd\nn^{\tt-\brdip}\ \ddd\nno^{-\tt}\\
&= \br_1\inv \, \ff\nn(\br_2\inv) \ \Ldots\ \ff\nn^\brdit(\br_\brdio\inv) \ \ddd\nn^\brdit\  \ddd\nn\, \br_\brdi\inv \ \ddd\nn^{\tt-\brdip}\ \ddd\nno^{-\tt}\\
&...\\
&= \br_1\inv \ \ddd\nn\,\br_2\inv\ \Ldots\ \ddd\nn\,\br_\brdio\inv \ \ddd\nn\,\br_\brdi\inv\ \ddd\nn^{\tt-\brdip}\ \ddd\nno^{-\tt}.
\end{align*}
As $\BKL\nno$ is a Garside monoid, there exists an integer $\kk$ such that
$\br_\brdi\inv\ddd\nno^\kk$ belongs to~$\BKL\nno$. Call the latter
braid~$\brr_\brdi$. Thus, we have 
$\ddd\nn\,\br_\brdi\inv=\ddd\nn\,\brr_\brdi\,\ddd\nno^{-\kk}$. Relation~\eqref{E:PushDdd} implies $\ddd\nn\,\brr_\brdi\,\ddd\nno^{-\kk}=\ff\nn(\brr_\brdi)\,\ddd\nn\,\ddd\nno^{-\kk}$. Then the braid $\br\inv\,\sep\nn\tt$ is equal to
\begin{equation}
\label{E:PUpper:1}
\br_1\inv\ \ddd\nn\,\br_2\inv\ \Ldots\ \ddd\nn\,\br_\brdio\inv \,\ff\nn(\brr_\brdi)\,\cdot\,\ddd\nn\,\ddd\nno^{-\kk}\,
\ddd\nn^{\tt-\brdip}\,\ddd\nno^{-\tt}.
\end{equation}
Each braid~$\br_\kk$ belongs to~$\BKL\nno$, and so its inverse~$\br_\kk\inv$ does not involve the $\nn$th strand.
Hence the left factor of \eqref{E:PUpper:1}, namely $\br_1\inv\ \ddd\nn\,\br_2\inv\ \Ldots\ \ddd\nn\,\br_\brdio\inv \,\ff\nn(\brr_\brdi)$, is $\sig\nno$-nonnegative. If $\brdi=\ttp$ holds, the right  factor of \eqref{E:PUpper:1}, namely $\ddd\nn\,\ddd\nno^{-\kk}\,\ddd\nn^{\tt-\brdip}\,\ddd\nno^{-\tt}$,  is equal to the braid $\ddd\nn\cdot\ddd\nno^{-\tt-\kk}$, which shows it is a $\sigg$-positive braid of type~$\aa1\nn$. 
If $\brdi\le\tt$ holds, \eqref{E:PUpper:1} ends with
$\ddd\nn\,\ddd\nno^{-\tt}$, which is a $\sigg$-positive braid of type $\aa1\nn$, and the
factor~$\ddd\nn\,\ddd\nno^{-\kk}\ddd\nn^{\tt-\brdi}$ is
$\sig\nno$-nonnegative. In each case, we conclude
using Lemma~\ref{L:TypeSigmaPos}$(iii)$ that  $\br\inv\sep\nn\tt$ is $\sigg$-positive of
type~$\aa1\nn$.
\end{proof}

Using the Key Lemma of Section~\ref{S:KeyLemma}, \ie, Proposition~\ref{P:KeyLemma}, we now establish  the implication~\eqref{E:Main2}.

\begin{prop}
\label{P:Lower}
For $\nn\ge 3$, the implication~\eqref{E:Main2} is true.
\end{prop}

\begin{proof}
Take $\br$ in $\BKL\nn$ and assume $\sep\nn\tt \SLLe\br$.  Let
$(\br_\brdi,\Ldots,\br_1)$ be the $\ff\nn$-splitting of~$\br$. By definition of $\SLL$,
the relation $\sep\nn\tt\SLLe\br$ implies $\tt \le \brdit$. Then $\sep\nn\tt\inv\,\br$ is equal to
\begin{equation}
\label{E:C:Lower}
\sep\nn\tt\inv\,\sep\nn\brdit\cdot\sep\nn\brdit\inv\br.
\end{equation}
By Lemma~\ref{L:OrderSep}, the factor $\sep\nn\tt\inv\,\sep\nn\brdit$ of~\eqref{E:C:Lower} is $\sigg$-positive or trivial.
By Lemma~\ref{L:SigsPositive},  the braid $\br_\brdi$ is $\sigg$-positive of type $\last{\br}_\brdi$.
Then, Proposition~\ref{P:KeyLemma} guarantees that the right factor
of~\eqref{E:C:Lower}, namely $\sep\nn\brdit\inv\,\br$, is $\sigg$-positive or trivial.
\end{proof}

\subsection{Proof of Theorem~\ref{T:coinOrdre}}

At this point, we know that the implications~\eqref{E:Main1} and~\eqref{E:Main2}
are true. It is not hard to deduce that the implication~\eqref{E:Main}, which is our
goal, is true when the breadth of~$\br$ is smaller than the breadth
of~$\brbr$, \ie, in the ``$\mathtt{Short}$''-case of the $\ShortLex$-ordering.

So, there remains to treat the ``$\mathtt{Lex}$''-case, \ie, the case
when $\br$ and $\brbr$ have the same $\nn$-breadth, and this is what we do
now. Actually, as was already mentioned, in
order to maintain an induction hypothesis, we shall prove a stronger
implication: instead of merely proving that the quotient braid~$\br\inv
\brbr$ is $\sigg$-positive, we shall prove the more precise conclusion
that $\br\inv \brbr$ is $\sigg$-positive of type $\aa\indi\nn$ for some~$\indi$ related
with the last letter of~$\brbr$.  That is why we shall consider the
``$\mathtt{Short}$''- and the ``$\mathtt{Lex}$''-cases simultaneously. 

\begin{prop}
\label{P:MainBis}
If $\br$ and $\brbr$ are nontrivial braids of~$\BKL\nn$, the relation $\br\SLL \brbr$
implies $\br<\nobreak \brbr$. Moreover, if the $\BKL\nno$-tail of $\brbr$
is trivial, then the braid~$\br\inv\brbr$ is $\sigg$-positive of type $\last{\brbr}$.
\end{prop}

\begin{proof}
 We use induction on~$\nn$.
For $\nn = 2$, everything is obvious,  as both~$<$ and $\SLL$ coincide with the standard
ordering of natural numbers. 

Assume $\nn \ge 3$, and $\br\SLL\brbr$ where $\br,\brbr$ belong to~$\BKL\nn$ and $\br \not= 1$ holds. Then $\brbr\not= 1$ holds as well.
Let $(\br_\brdi,\Ldots,\br_1)$ and 
$(\brbr_\brdii,\Ldots,\brbr_1)$ be the $\ff\nn$-splittings of $\br$ and $\brbr$.   As
$\br\SLL\brbr$ holds, we have $\brdi\le\brdii$.
Write $\br_\brdii=\Ldots=\br_\brdip=1$.
Let $\tt$ be the maximal integer in $\{1,\Ldots,\brdii\}$ satisfying $\br_\tt\SLL\brbr_\tt$.
By definition of $\SLL$, such a $\tt$ exists.
Write $\brbrr_\tt=\br_\tt\inv\,\brbr_\tt$.
By induction hypothesis, the braid $\brbrr_\tt$ is $\sigg$-positive.
Moreover, if $\tt\ge2$ holds, then the braid $\brbrr_\tt$ is $\sigg$-positive of type $\last{\brbr}_\tt$.

Assume $\tt=1$.
Then the braid $\br\inv\brbr$ is equal to $\brbrr_\tt$.
Hence, it is $\sigg$-positive.
As the $\BKL\nno$-tail of $\brbr_1$ is non-trivial, we have nothing more to prove.

Assume now $\tt\ge2$.
Let $\aa\indii\nn$ be the last letter of $\ff\nn^\tto(\brbr_\tt)\cdot\,\ldots\cdot\ff\nn(\br_2)$.
The sequence $(\br_\tto,\Ldots,\br_1)$ is a $\ff\nn$-splitting of a braid of $\nn$-breadth $\tto$.
Then Proposition~\ref{P:Upper} implies that the braid $\brr$, that is equal to 
\[\br_1\inv\cdot\,\Ldots\cdot\ff\nn^\ttp(\br_\tto\inv)\cdot\sep\nn\ttt,\] is $\sigg$-positive of type $\aa1\nn$.
Let $\brbrr$ be the braid 
\[\sep\nn\ttt\inv\cdot\ff\nn^\tto(\brbrr_\tt)\cdot\ff\nn^\ttt(\brbr_\tto)\cdot\,\Ldots\cdot\ff\nn(\brbr_2).\]
As $\brbrr_\tt$ is $\sigg$-positive of type $\last{\brbr}_\tt$, Proposition~\ref{P:KeyLemma} implies that the braid
$\brbrr$ is $\sigg$-positive of type $\aa\indii\nn$ or trivial (the latter occurs only for $\indii=1)$.
Then, in any case, the braid~$\brr\,\brbrr$ is $\sigg$-positive of type $\aa\indii\nn$.
As, by construction, we have $\br\inv\brbr=\brr\cdot\brbrr\cdot\brbr_1$, the braid $\br\inv\brbr$ is $\sigg$-positive.
Moreover, assume that the $\BKL\nno$-tail of $\brbr$ is trivial.
Then $\brbr_1$ is trivial and the braid $\brbr$ ends with $\aa\indii\nn$.
In this case we have $\br\inv\brbr=\brr\cdot\brbrr$, a $\sigg$-positive braid of type $\aa\indii\nno$.
Therefore $\br\inv\brbr$ is a $\sigg$-positive braid of type $\last{\brbr}$.
\end{proof}

Our proof of Proposition~\ref{P:MainBis} is therefore complete, and so is the proof of
Theorem~\ref{T:coinOrdre}, which is strictly included in the latter.

So, we now have a complete description of the restriction of the Dehornoy ordering
of braids to the Birman--Ko--Lee monoid~$\BKL\nn$. The characterization of Theorem~1
is inductive, connecting the ordering on~$\BKL\nn$ to the ordering on~$\BKL\nno$.
Actually, it is very easy to obtain a non-inductive formulation. Indeed, we can define the
\emph{iterated splitting}~$T(\br)$ of a braid~$\br$ of~$\BKL\nn$ to be the tree obtained
by substituting the $\ff\nno$-splittings of the entries in the $\ff\nn$-splitting, and
iterating with the $\ff\nnt$-splittings, and so on until we reach~$\BKL2$, \ie, the natural
numbers. In this way, we associate with every braid~$\br$ of~$\BKL\nn$ a tree~$T(\br)$
with branches of length~$\nnt$ and natural numbers labeling the
leaves---see~\cite{Dehornoy:AF} for an analogous construction. Then Theorem~1
immediately implies that, for $\br, \brbr$ in~$\BKL\nn$, the relation $\br < \brbr$
holds if and only if the tree~$T(\br)$ is $\ShortLex$-smaller than the tree~$T(\brbr)$.

\begin{rema}
Whether the tools developed in~\cite{F:SR} and in the current may be
adapted to the case of~$\BP\nn$ is an open question. The starting point of our approach is
very similar to that of~\cite{Dehornoy:AF} and~\cite{Burckel:WO}. However, it seems that the
machinery of ladders and dangerous braids involved in the technical results of
Section~3 are really specific to the case of the dual monoids, and heavily depend on the
highly redundant character of the relations connecting the Birman--Ko--Lee generators.

By contrast, a much more promising approach would be to investigate the restriction
of the finite Thurston-type braid orderings of~\cite{ShortWiest} to the
monoids~$\BKL\nn$ along the lines developed in~\cite{Ito}. In particular, it should be
possible to determine the isomorphism type explicitly.
\end{rema}

%

%
%
%
%
%
%

\bibliographystyle{ams-pln} 
\bibliography{Bibliography}

 \end{document}